\documentclass[11pt]{article}

\usepackage{amsmath, amssymb, amsthm} 
\usepackage{centernot}
\usepackage{mathpazo}
\usepackage[margin=2.5cm]{geometry}
\usepackage{color, graphicx, soul}
\usepackage{caption}
\usepackage{subcaption}
\usepackage{tikz}
\usepackage{pgfplots}
\pgfplotsset{compat=1.11}

\usepackage{hyperref}
\hypersetup{
	colorlinks=true,        
	linkcolor=blue,         
	citecolor=blue,         
	urlcolor=blue           
}

\usepackage{empheq}
\definecolor{myblue}{rgb}{.9, .9, 1}
\newcommand*\mybluebox[1]{%
	\colorbox{myblue}{\hspace{1em}#1\hspace{1em}}
}

\usepackage[capitalize]{cleveref}
\crefname{equation}{}{equations}
\crefname{chapter}{Appendix}{chapters}
\crefname{item}{}{items}
\crefname{figure}{Figure}{figures}
\makeatletter
\def\namedlabel#1#2{\begingroup
   \def\@currentlabel{#2}%
   \label{#1}\endgroup
}
\makeatother

\makeatletter
\def\th@plain{%
  \thm@notefont{}
  \itshape 
}
\def\th@definition{%
  \thm@notefont{}
  \normalfont 
}
\makeatother

\newtheorem{theorem}{Theorem}[section]
\newtheorem{lemma}[theorem]{Lemma}
\newtheorem{corollary}[theorem]{Corollary}
\newtheorem{proposition}[theorem]{Proposition}
\newtheorem{fact}[theorem]{Fact}
\newtheorem{assumption}[theorem]{Assumption}
\theoremstyle{definition}

\theoremstyle{definition}
\newtheorem{example}[theorem]{Example}
\theoremstyle{definition}
\newtheorem{remark}[theorem]{Remark}

\usepackage[shortlabels]{enumitem}
\setlist[enumerate]{nosep}

\allowdisplaybreaks

\newcommand{\menge}[2]{\{{#1}~\big |~{#2}\}} 
\newcommand{\mmenge}[2]{\bigg\{{#1}~\bigg |~{#2}\bigg\}} 
\newcommand{\Menge}[2]{\left\{{#1}~\Big|~{#2}\right\}} 

\newcommand{\ball}[2]{\mathbb{B}\left({#1};{#2}\right)}
\newcommand{\scal}[2]{\left\langle {#1},{#2} \right\rangle}
\newcommand{\lev}[1]{\operatorname{lev}_{\leq #1}}

\newcommand{\To}{\ensuremath{\rightrightarrows}}

\newcommand{\NN}{\ensuremath{\mathbb N}}
\newcommand{\nnn}{\ensuremath{{n\in{\mathbb N}}}}

\newcommand{\RR}{\ensuremath{\mathbb R}}
\newcommand{\RP}{\ensuremath{\mathbb{R}_+}}
\newcommand{\RPP}{\ensuremath{\mathbb{R}_{++}}}
\newcommand{\RM}{\ensuremath{\mathbb{R}_-}}
\newcommand{\RMM}{\ensuremath{\mathbb{R}_{--}}}

\newcommand{\argmin}{\ensuremath{\operatorname*{argmin}}}
\newcommand{\Limsup}{\ensuremath{\operatorname*{Limsup}}}

\newcommand{\inte}{\ensuremath{\operatorname{int}}}

\newcommand{\ran}{\ensuremath{\operatorname{ran}}}

\newcommand{\dom}{\ensuremath{\operatorname{dom}}}
\newcommand{\intdom}{\ensuremath{\operatorname{int}\operatorname{dom}}\,}
\newcommand{\epi}{\ensuremath{\operatorname{epi}}}
\newcommand{\gra}{\ensuremath{\operatorname{gra}}}
\newcommand{\Fix}{\ensuremath{\operatorname{Fix}}}
\newcommand{\Id}{\ensuremath{\operatorname{Id}}}

\newcommand{\sgn}{\ensuremath{\operatorname{sgn}}}

\begin{document}

\title{A Lyapunov-type approach to convergence of the Douglas--Rachford algorithm}

\author{
Minh N.\ Dao\thanks{CARMA, 
					University of Newcastle, 
					Callaghan, NSW 2308, Australia.
					E-mail:~\href{mailto:daonminh@gmail.com}{daonminh@gmail.com}}
\and
Matthew K.\ Tam\thanks{Institut f\"ur Numerische und Angewandte Mathematik,
					Universit\"at G\"ottingen,
					37083 G\"ottingen, Germany.
					\mbox{E-mail:}~\href{mailto:m.tam@math.uni-goettingen.de}{m.tam@math.uni-goettingen.de}}
}

\date{\today}

\maketitle

\begin{abstract}
The Douglas--Rachford projection algorithm is an iterative method used to find a point in the intersection of closed constraint sets. The algorithm has been experimentally observed to solve various nonconvex feasibility problems which current theory cannot sufficiently explain. In this paper, we prove convergence of the Douglas--Rachford algorithm in a potentially nonconvex setting. Our analysis relies on the existence of a Lyapunov-type functional whose convexity properties are not tantamount to convexity of the original constraint sets. Moreover, we provide various nonconvex examples in which our framework proves global convergence of the algorithm.
\end{abstract}

{\small
\noindent
{\bfseries Mathematics Subject Classification (MSC 2010):}
	90C26 $\cdot$ 
    47H10 $\cdot$ 
    37B25 

}

\section{Introduction}

The \emph{Douglas--Rachford algorithm (DRA)} is an iterative method used to solve the so-called  \emph{feasibility problem} which asks for a point in the intersection of closed constraint sets. The method generates a sequence by combining the \emph{nearest point projectors} of the individual constraint sets with exploiting the structure of problems in which these individual projectors can be efficiently computed or, at least, more efficiently than a \emph{direct} attempt to solve the original problem. The origins of the method can be traced to work of Douglas~\&~Rachford \cite{DR56} where it was proposed as a method for numerically solving problems arising in heat conduction. In the convex setting, the situation is fairly well understood; convergence is due to Lions~\&~Mercier \cite{LM79} and has since been refined in various works \cite{BCL04,BDM16,BM17,Sva11}.
	
In the absence of convexity, Borwein~\&~Sims \cite{BS11} established \emph{local convergence} of the DRA applied to a prototypical nonconvex feasibility problem involving a line and sphere. Here ``prototypical'' is meant in the sense of being an accessible model for imaging problems where phase is to be reconstructed from magnitude measurements whilst retaining the mathematical complexities. The same prototype has since studied in \cite{AB13,Gil16}. In their paper, Borwein~\&~Sims \cite{BS11} conjectured that the DRA was actually \emph{globally convergent}; a conjecture that was recently resolved in the affirmative by Benoist \cite{Ben15} through a cleverly constructed Lyapunov function. Global convergence of the DRA for prototypical combinatorial optimization problems has been also proven in \cite{ABT16,BD16}.

A general approach to convergence of the DRA without convexity was provided by Phan~\cite{Pha16}, to which work of Luke~\&~Hesse \cite{HL13} was a precursor. This approach follows related works, originating from \cite{LLM09}, which focus on the \emph{method of alternating projections} and assume that \emph{local regularity} properties of the underlying constraint sets hold near solutions of the problem (see also \cite{NR16}). The main difficulty in applying these results, lies in that they give little information regarding the \emph{region of convergence}, that is, the starting points from which the algorithm converges. Moreover, in practice, finding a point sufficiently close to a solution of a feasibility problem is often just as difficult as solving the original feasibility problem itself.

In this work, we generalize Benoist's approach to construction of Lyapunov-type functionals as a tool to prove convergence of the DRA. In particular, we show that convergence of the DRA is ensured provided that the constructed Lyapunov-type function possess appropriate convexity properties. We emphasize here that the convexity properties of the Lyapunov-type function are \emph{independent} of convexity of the underlying feasibility problem. As a consequence of our analysis, a region of convergence of the DRA can be identified by analyzing the Lyapunov-type function associated with the problem at hand.

The remainder of this paper is organized as follows. In Section~\ref{s:prelim}, we introduce the necessary notions of nonsmooth analysis. In Section~\ref{s:main}, we give a precise description of the Douglas--Rachford operator. In Section~\ref{s:stability of DRA} we provide conditions under which the DRA enjoys stability properties near fixed points. Our Lyapunov approach to convergence of the algorithm follows in Section~\ref{s:lyapunov}. Examples to which the results apply are considered in Section~\ref{s:examples} together with some counter-examples to demonstrate that both the method of alternating projections and \emph{Newton's method} can fail to converge to a solution even when the DRA does. In fact, global convergence of the DRA is obtained in all bar one of our provided examples.

\section{Preliminaries}
\label{s:prelim}
In this section we introduce and recall necessary notions and tools from nonsmooth analysis. Throughout this work, we assume that
\begin{empheq}[box=\mybluebox]{equation}
\text{$X$ is a Euclidean space,}
\end{empheq}
(\emph{i.e.,} a finite-dimensional real Hilbert space) with inner product $\scal{\cdot}{\cdot}$ and induced norm $\|\cdot\|$. Given two real Hilbert spaces $X$ and $Y$ with corresponding inner products denoted $\scal{\cdot}{\cdot}_X$ and $\scal{\cdot}{\cdot}_Y$, where appropriate, we use the product space $X\times Y$ which is a Hilbert space when equipped with the inner product defined by 
 \begin{equation}
 \scal{(x_1, y_1)}{(x_2, y_2)}_{X\times Y} := \scal{x_1}{x_2}_X +\scal{y_1}{y_2}_Y.
 \end{equation}
We denote the set of nonnegative integers is by $\NN$ and the set of real numbers by $\RR$. The set of nonnegative real numbers is denoted $\RP := \menge{x \in \RR}{x \geq 0}$ and set of the positive real numbers $\RPP := \menge{x \in \RR}{x >0}$. The sets of nonpositive and negative real numbers, denoted $\RM$ and $\RMM$ respectively, are defined analogously. Given a subset $C$ of $X$, its closure and interior are denoted respectively by $\overline C$ and $\inte C$. For a point $x \in X$ and scalar $\rho \in \RPP$, the closed ball centered at $x$ with radius $\rho$ is denoted $\ball{x}{\rho} :=\menge{y \in X}{\|x -y\| \leq \rho}$.

\subsection{The Douglas--Rachford algorithm}
In this section, we recall the background material for the Douglas--Rachford algorithm. 
Let $C$ be a nonempty subset of $X$. The \emph{projector} onto $C$ is the mapping
\begin{equation}
P_C\colon X \To C\colon x \mapsto \argmin_{c \in C} \|x -c\| =\menge{c \in C}{\|x -c\| =d_C(x)},
\end{equation}
where $d_C(x) :=\inf_{c \in C} \|x -c\|$ is the \emph{distance} from $x$ to $C$. 
Each $y \in P_C(x)$ is a \emph{nearest point} of $x$ in $C$, and called a \emph{projection} of $x$ onto $C$.
Since we consider only finite-dimensional spaces $X$, closedness of the set $C$ is necessary and sufficient for $C$ being \emph{proximinal}, 
\emph{i.e.,} $(\forall x \in X)$ $P_Cx \neq\varnothing$; see \cite[Corollary~3.13]{BC11}. 
In an abuse of notation, we write $P_Cx =c$ whenever $P_Cx =\{c\}$.

Let $C$ and $D$ be closed subsets of $X$ such that $C\cap D \neq\varnothing$. 
The \emph{feasibility problem} is
\begin{equation}
\label{e:prob}
\text{find a point in $C\cap D$.}
\end{equation}
A classical splitting method for solving \eqref{e:prob} is the so-called \emph{Douglas--Rachford algorithm} 
which is concisely described as the fixed point iteration corresponding to the \emph{Douglas--Rachford (DR) operator} defined by
\begin{empheq}[box=\mybluebox]{equation}
T_{C,D} :=\tfrac{1}{2}(\Id +R_DR_C),
\end{empheq}
where $\Id$ is the identity operator, 
and $R_C :=2P_C -\Id$ and $R_D :=2P_D -\Id$ are the \emph{reflectors} across $C$ and $D$, respectively. A sequence $(x_n)_\nnn$ is called a \emph{DR sequence (with respect to $(C, D)$)}, with starting point $x_0 \in X$, if
\begin{empheq}[box=\mybluebox]{equation}
\label{e:DRAseq}
(\forall\nnn)\quad x_{n+1} \in T_{C,D}x_n,
\end{empheq}
where we note that
\begin{equation}
\label{e:Tx}
(\forall x\in X)\quad T_{C,D}x =\frac{1}{2}(\Id +R_DR_C)x =\menge{x -c +P_D(2c -x)}{c\in P_Cx}.
\end{equation}
In the literature, the DRA for feasibility problems is also known as \emph{averaged alternating reflections} \cite{BCL04} and \emph{reflect-reflect-average} method \cite{BS11}. For other connections, we refer the reader to \cite{BCL02}.

The following fact gives important properties of convex projectors.
\begin{fact}[Projectors and reflectors onto convex sets]
\label{f:proj}
Let $C$ be a nonempty closed convex subset of $X$. Then the following hold: 
\begin{enumerate}
\item 
\label{f:proj_P}
$P_C$ is everywhere single-valued and \emph{firmly nonexpansive}, that is, 
\begin{equation}
(\forall x \in X)(\forall y \in X)\; \|P_Cx -P_Cy\|^2 +\|(\Id -P_C)x -(\Id -P_C)y\|^2 \leq \|x -y\|^2.
\end{equation}
\item
\label{f:proj_R}
$R_C$ is everywhere single-valued and \emph{nonexpansive}, that is, 
\begin{equation}
(\forall x \in X)(\forall y \in X)\; \|R_Cx -R_Cy\| \leq \|x -y\|.
\end{equation}
\end{enumerate}
In particular, $P_C$ and $R_C$ are continuous on $X$.
\end{fact}
\begin{proof}
\ref{f:proj_P}: See \cite[Theorem~3.14~\&~Proposition~4.8]{BC11}.
\ref{f:proj_R}: This follows from \ref{f:proj_P} and \cite[Corollary~4.10]{BC11}.
\end{proof}

In the case that one of the constraints is convex, we make the following observations. In what follows, recall that a sequence $(x_n)_\nnn$ is \emph{asymptotically regular} if $x_n -x_{n+1} \to 0$ as $n \to +\infty$.
\begin{lemma}[Properties of the DRA]
\label{l:1cvx}
Let $C$ be a closed convex subset and $D$ be a closed subset of $X$ such that $C\cap D \neq\varnothing$,
and let $(x_n)_\nnn$ be a DR sequence with respect to $(C, D)$.
Then the following hold:
\begin{enumerate}
\item 
\label{l:1cvx_single}
$T_{C,D} =\frac{1}{2}(\Id +R_DR_C) =\Id -P_C +P_DR_C.$
\item
\label{l:1cvx_stable}
$T_{C,D}(\ball{\bar{x}}{\delta}) \subseteq \ball{\bar{x}}{2\delta}$ whenever $\bar{x} \in C\cap D$ and $\delta \in \RPP$.
\item
\label{l:1cvx_cluster}
If $(x_n)_\nnn$ is asymptotically regular and possess a cluster point $x$, then $P_Cx \in C\cap D$.   
\end{enumerate}
\end{lemma}
\begin{proof}
\ref{l:1cvx_single}: Combine \eqref{e:Tx} with the single-valuedness of $P_C$ (\cref{f:proj}).

\ref{l:1cvx_stable}: Let $\bar{x} \in C\cap D$, let $\delta \in \RPP$, let $x \in \ball{\bar{x}}{\delta}$, and let $x_+ \in T_{C,D}x$.
Then there exists $p \in P_DR_Cx$ such that $x_+ =\frac{1}{2}(x +(2p -R_Cx))$.
Since $\bar{x} \in C\cap D$, it follows that $R_C\bar{x} =\bar{x}$ and thus
\begin{subequations}
\begin{align}
2\|x_+ -\bar{x}\| &=\|(x -\bar{x}) +2(p -R_Cx) +(R_Cx -\bar{x})\| \\
&\leq \|x -\bar{x}\| +2\|p -R_Cx\| +\|R_Cx -\bar{x}\| \\
&=\|x -\bar{x}\| +2d_D(R_Cx) +\|R_Cx -\bar{x}\| \\
&\leq \|x -\bar{x}\|+ 2\|R_Cx-\bar{x}\|+\|R_Cx -\bar{x}\| \\
&=\|x -\bar{x}\| +3\|R_Cx -R_C\bar{x}\| \leq 4\|x-\bar{x}\|, 
\end{align}
\end{subequations}
where the last estimate follows from the nonexpansiveness of $R_C$ (\cref{f:proj}).
Altogether, we obtain that $\|x_+ -\bar{x}\| \leq 2\|x -\bar{x}\| \leq 2\delta$, hence $x_+ \in \ball{\bar{x}}{2\delta}$ and the result follows.

\ref{l:1cvx_cluster}: Using \ref{l:1cvx_single} yields
\begin{equation}
\label{e:single'}
(\forall\nnn)\quad p_n :=x_{n+1} -x_n +P_Cx_n \in P_DR_Cx_n \subseteq D.
\end{equation}
Let $x$ be a cluster point of $(x_n)_\nnn$. 
Then there exists a subsequence $(x_{k_n})_\nnn$ of $(x_n)_\nnn$ such that $x_{k_n} \to x$.
By \cref{f:proj}, $P_C$ is continuous and so $P_Cx_{k_n} \to P_Cx$.
Combining with \eqref{e:single'} and the asymptotic regularity of $(x_n)_\nnn$, this gives $p_{k_n} \to P_Cx$. 
Noting that $(\forall\nnn)$ $p_{k_n} \in D$ and that $D$ is closed, we deduce that $P_Cx \in D$ and therefore $P_Cx \in C\cap D$.
\end{proof}

\subsection{Convexity}
Given an extended-real-valued function $f\colon X \to \left[-\infty, +\infty\right]$, 
its \emph{effective domain} is denoted by $\dom f :=\menge{x\in X}{f(x)<+\infty}$, 
its \emph{graph} by $\gra f :=\menge{(x, \rho)\in X\times \RR}{f(x) =\rho}$, 
its \emph{epigraph} by $\epi f :=\menge{(x, \rho)\in X\times \RR}{f(x) \leq \rho}$,
and its \emph{lower level set} at height $\xi \in \RR$ by $\lev{\xi}f :=\menge{x \in X}{f(x) \leq \xi}$.
The function $f$ is said to be \emph{proper} if $\dom f\neq\varnothing$ and it never takes the value $-\infty$, 
\emph{lower semicontinuous (lsc)} if $f(x) \leq \liminf_{y \to x} f(y)$ for every $x \in \dom f$,
and \emph{convex} if
\begin{multline}
\label{e:cvx}
\hspace{1.5cm}(\forall x \in \dom f)(\forall y \in \dom f)(\forall \lambda \in \left]0, 1\right[)\\
f((1 -\lambda)x +\lambda y) \leq (1 -\lambda)f(x) +\lambda f(y).\hspace{1.5cm}
\end{multline}
Let $f\colon X \to \left[-\infty, +\infty\right]$ be proper. 
Then $f$ is said to be \emph{strictly convex} if, in addition to being convex, the inequality in \eqref{e:cvx} is strict whenever $x\neq y$.
We say that $f$ is \emph{convex on $C$} (respectively \emph{strictly convex on $C$}) 
if the corresponding inequality holds whenever $x \in C$ and $y \in C$.  
ity.

\begin{fact}
\label{f:cvx}
Let $f\colon X \to \left[-\infty, +\infty\right]$ be a proper function and let $C$ be a nonempty open convex subset of $\dom f$.
\begin{enumerate}
\item 
\label{f:cvx_d1}
Suppose that $f$ is G\^ateaux differentiable on $C$. Then the following hold:
	\begin{enumerate}
	\item 
	$f$ is convex on $C$ if and only if $\nabla f$ is \emph{monotone} on $C$ in the sense that
	\begin{equation}
	(\forall x \in C)(\forall y \in C)\quad \scal{x -y}{\nabla f(x) -\nabla f(y)} \geq 0.
	\end{equation}
	\item
	$f$ is strictly convex on $C$ if and only if $\nabla f$ is \emph{strictly monotone} on $C$ in the sense that
	\begin{equation}
	(\forall x \in C)(\forall y \in C)\quad x\neq y \;\implies\; \scal{x -y}{\nabla f(x) -\nabla f(y)} >0.
	\end{equation}
	\end{enumerate}
\item
\label{f:cvx_d2}
Suppose that $f$ is twice G\^ateaux differentiable on $C$. Then the following hold:
	\begin{enumerate}
	\item 
	$f$ is convex on $C$ if and only if $\nabla^2 f(x)$ is \emph{positive semidefinite} for every $x \in C$.
	\item
	$f$ is strictly convex on $C$ if $\nabla^2 f(x)$ is \emph{positive definite} for every $x \in C$.
	\end{enumerate}
\end{enumerate} 
\end{fact}
\begin{proof}
This follows from \cite[Propositions~17.10~\&~17.13]{BC11}.
\end{proof}

\subsection{Subdifferentiability}
The \emph{limiting normal cone} to a subset $C$ of $X$ at a point $x \in X$ is defined by
\begin{equation}
N_C(x) :=\mmenge{x^*\in X}{\exists \varepsilon_n\downarrow 0,\; x_n\stackrel{C}{\to} x, 
\text{~and~} x_n^* \to x^* \text{~with~} \limsup_{y\stackrel{C}{\to} x_n} \frac{\scal{x_n^*}{y -x_n}}{\|y -x_n\|} \leq \varepsilon_n}
\end{equation}
if $x \in C$, and by $N_C(x) :=\varnothing$ otherwise. 
Here the notation $y \stackrel{C}{\to} x$ means $y \to x$ with $y \in C$.

Let $f\colon X \to \left[-\infty, +\infty\right]$, let $x \in X$ with $|f(x)| <+\infty$, and let $\varepsilon \in \RP$. 
The \emph{limiting subdifferential} of $f$ at $x$ is given by
\begin{equation}
\partial f(x) :=\menge{x^*\in X}{(x^*, -1)\in N_{\epi f}(x, f(x))}
\end{equation}
and the \emph{analytic $\varepsilon$-subdifferential} of $f$ at $x$ is given by
\begin{equation}
\widehat{\partial}_{\varepsilon} f(x) :=\mmenge{x^* \in X}{\liminf_{y \to x}\frac{f(y) -f(x) -\scal{x^*}{y -x}}{\|y -x\|} \geq -\varepsilon}.
\end{equation}
Both subdifferentials of $f$ at a point $x$ are defined to be empty when $|f(x)| =+\infty$.
The limiting subdifferential can be represented in analytic form \cite[Theorem~1.89]{Mor06}
\begin{equation}
\label{e:lim_eps_subdiff}
\partial f(x) =\Limsup_{\substack{y \stackrel{f}{\to} x,\, \varepsilon\,\downarrow\, 0}} \widehat{\partial}_{\varepsilon} f(y) 
=\menge{x^* \in X}{\exists\varepsilon_n \to 0,\, x_n \stackrel{f}{\to} x,\, x_n^*\to x^* \text{~with~} x_n^*\in \widehat{\partial}_{\varepsilon_n}f(x_n)},
\end{equation}
where the notation $y \stackrel{f}{\to} x$ means $y \to x$ with $f(y) \to f(x)$ and 
\begin{equation}
\Limsup_{y \to x} F(y) :=\menge{x^* \in X}{\exists x_n \to x,\, x_n^* \to x^* \text{~with~} x_n^*\in F(x_n)}
\end{equation}
denotes the \emph{sequential Painlev\'e--Kuratowski upper limit} of $F$ at $x$.

We now recall some important properties of the limiting subdifferential.
\begin{fact}[Fermat's rule]
\label{f:Fermat}
Suppose that a function $f\colon X \to \left[-\infty, +\infty\right]$ attains a local minimum at a point $x$ with $|f(x)| <+\infty$. Then $0 \in \partial f(x)$.
\end{fact}
\begin{proof}
This follows from \cite[Proposition~1.114]{Mor06}.
\end{proof}
\begin{fact}[Sum and product rules]
\label{f:calculus}
Consider two functions, $f\colon X \to \left[-\infty, +\infty\right]$ and $g\colon X \to \left[-\infty +\infty\right]$, and let $x \in X$. 
The following assertions hold.
\begin{enumerate}
\item 
\label{f:calculus_sum}
If $f$ is finite at $x$ and $g$ is strictly differentiable at $x$, then 
\begin{equation}
\partial(f +g)(x) =\partial f(x) +\nabla g(x).
\end{equation}
\item
\label{f:calculus_product}
If $f$ and $g$ are Lipschitz continuous around $x$, then 
\begin{equation}
\partial(f\cdot g)(x) =\partial(g(x)f +f(x)g)(x) \subseteq \partial(g(x)f)(x) +\partial(f(x)g)(x).
\end{equation}
\end{enumerate}
\end{fact}
\begin{proof}
\ref{f:calculus_sum}: \cite[Proposition~1.107]{Mor06}.
\ref{f:calculus_product}: \cite[Propositions~1.111 and~3.45]{Mor06}.
\end{proof}
If $f$ is lsc around $x$, then a convenient, and often used, representation for the limiting subdifferential is given by \cite[Theorems~1.89~\&~2.34]{Mor06}
\begin{equation}
\label{e:lim_subdiff}
\partial f(x) =\Limsup_{\substack{y \stackrel{f}{\to} x}} \widehat{\partial} f(y) 
=\menge{x^*\in X}{\exists x_n \stackrel{f}{\to} x \text{~and~} x_n^* \to x^* 
\text{~with~} x_n^* \in \widehat{\partial} f(x_n)},
\end{equation} 
where $\widehat{\partial} f :=\widehat{\partial}_0 f$ is the so-called \emph{Fr\'echet subdifferential} of $f$. However, in what follows, it will be necessary to consider the subdifferentials of both $f$ and $-f$ simultaneously. In this case, \eqref{e:lim_subdiff}  cannot not be applied because $f$ and $-f$ are usually  not simultaneously lsc (\emph{e.g.,} if $f$ takes the value $+\infty$). Further, it is worth emphasizing, that $\partial f$ and $-\partial(-f)$ are, in general, considerably different. For instance, the function $f =|\cdot|\colon \RR \to \RR$ has
\begin{equation}
\partial f(0) =\left[-1, 1\right] \quad\text{and}\quad -\partial(-f)(0) =\{-1, 1\}.
\end{equation}
Combining these two subdifferentials yields the \emph{symmetric subdifferential} of $f$ which is defined by
\begin{equation}
\label{e:symdiff}
\partial^0\! f :=\partial f\cup \partial^+\! f,
\end{equation}
where $\partial^+\! f := -\partial(-f)$ is the so-called \emph{limiting upper subdifferential} of $f$.
In contrast to the limiting subdifferential, the symmetric subdifferential possess the classical ``plus-minus'' symmetry (\emph{i.e.,}~$\partial^0\!(-f) =-\partial^0\! f$).
Also note that, if $f$ is strictly differentiable at $x$, then \cite[Corollary~1.82]{Mor06}
\begin{equation}
\partial f(x) =\partial^+\! f(x) =\partial^0\! f(x) =\{\nabla f(x)\}.
\end{equation} 
If $f$ is convex, then the limiting subdifferential reduces to the \emph{convex subdifferential (or Fenchel subdifferential)} of convex analysis \cite[Theorem~1.93]{Mor06}, that is,
\begin{equation}
\label{e:Fensubdiff}
\partial f(x) =\menge{x^*\in X}{(\forall y\in X)\;\scal{x^*}{y -x} \leq f(y) -f(x)},
\end{equation}
and we have the inclusions
\begin{equation}
\label{e:cvx_inclus}
\partial^+\! f(x) \subseteq \partial f(x) =\partial^0\! f(x).
\end{equation}

The following property for the limiting subdifferential is mentioned without proof in \cite{Mor06,RW98} and we therefore we provide one for the convenience of the reader. Furthermore, note that lower semicontinuity is not assumed and so we cannot simply appeal to the representation \eqref{e:lim_subdiff}.
\begin{lemma}[Scalar multiplication rule]
\label{l:scalar}
Let $f\colon X \to \left[-\infty, +\infty\right]$. Then
\begin{equation}
\label{e:scalar_subdiff}
(\forall \lambda \in \RR\smallsetminus\{0\})\quad
\partial(\lambda f) =
\begin{cases}
\lambda\partial f &\text{~if~} \lambda > 0,\\ 
\lambda\partial^+\! f &\text{~if~} \lambda <0.
\end{cases}
\end{equation}
\end{lemma}
\begin{proof}	
Let $\lambda \in \RR\smallsetminus\{0\}$. Then $|f(x)|<+\infty$ if and only if $|\lambda f(x)|<+\infty$ for all $x\in X$. In particular, this shows that \eqref{e:scalar_subdiff} holds at points at which $f$ is not finite. Assume now that $x\in X$ with $|f(x)|<+\infty$. By the definition of the analytic $\varepsilon$-subdifferential of $f$,
\begin{equation}
\widehat{\partial}_{\varepsilon}(\lambda f)(x) 
     = \begin{cases}
         \lambda\widehat{\partial}_{\lambda\varepsilon}f(x)  & \text{if }\lambda>0, \\
         -\lambda\widehat{\partial}_{\varepsilon}(-f)(x)  & \text{if }\lambda<0. 
       \end{cases}
\end{equation}
  Hence, if $\lambda>0$, then as $|f(x)| <+\infty$, we may apply \eqref{e:lim_eps_subdiff} to deduce that
\begin{equation}
\partial(\lambda f)(x) 
         = \Limsup_{\substack{y \stackrel{f}{\to} x,\, \varepsilon\,\downarrow\, 0}} \widehat{\partial}_{\varepsilon}(\lambda f)(y) 
         = \lambda\Limsup_{\substack{y \stackrel{f}{\to} x,\, \varepsilon\,\downarrow\, 0}} \widehat{\partial}_{\varepsilon/\lambda}f(y)
         = \lambda\partial f(x).
\end{equation}
The argument for $\lambda<0$ is performed analogously.
\end{proof}

\begin{remark}[Multiplication by zero]
 Care must be exercised in the case that $\lambda=0$ in \cref{l:scalar}.  
  Consider, for instance, the lsc convex function $f\colon \RR \to \left[-\infty, +\infty\right]$ defined by 
   \begin{equation}
    f(x) := \begin{cases}
    	        -\sqrt{1 -x^2} &\text{~if~}|x| \leq 1, \\
    	        +\infty &\text{~otherwise},
    	    \end{cases}
   \end{equation}  
   which has $\partial f(\pm 1)=\varnothing$. Under the convention that $0\cdot(+\infty)=+\infty$, it follows that $0\cdot f=\iota_{\left[-1, 1\right]}$, where $\iota_C$ is the indicator function of a set $C$, so that $\partial (0\cdot f) =N_{\left[-1, 1\right]}$ and
\begin{equation}
\partial (0\cdot f)(-1) =\left]-\infty, 0\right] \quad\text{and}\quad \partial (0\cdot f)(1) =\left[0, +\infty\right[.
\end{equation}
   Alternatively, under the convention that $0\cdot(+\infty) =0 =0\cdot(-\infty)$ as suggested in \cite[Section~1E]{RW98}, we have $0\cdot f=0$ and hence that $\partial(0\cdot f)(\pm 1) =\partial(0)(\pm 1) = \{0\}$. 
  
  For our purposes, both conventions are problematic, and thus we shall treat the cases of $\lambda=0$ directly as it arises.
\end{remark}

As holds for the limiting subdifferential, the symmetric subdifferential also enjoys the following robustness property.
\begin{lemma}[Robustness of the symmetric subdifferential]
\label{l:robustness}
Let $f\colon X \to [-\infty, +\infty]$ and let $ x \in X$ with $|f(x)| <+\infty$. 
Then the symmetric subdifferential has the following \emph{robustness property}
\begin{equation}
\partial^0\! f(x) =\Limsup_{y\stackrel{f}{\to}x} \partial^0\! f(y) 
=\menge{x^* \in X}{\exists x_n\stackrel{f}{\to} x \text{~and~} x_n^* \to x^* \text{~with~} x_n^* \in \partial^0\! f(x_n)}. 
\end{equation}
\end{lemma}
\begin{proof}
It is clear that 
\begin{equation}
\partial^0\! f(x) \subseteq \Limsup_{y\stackrel{f}{\to}x} \partial^0\! f(y). 
\end{equation}
To prove the opposite inclusion, we assume that $x_n\stackrel{f}{\to} x$ and $x_n^* \to x^*$ with $x_n^* \in \partial^0\! f(x_n)$. 
Since $\partial^0\! f =\partial f\cup \partial^+\! f$, by passing to subsequences if necessary, 
it suffices to prove the results assuming that the sequence $(x_n^*)_{\nnn}$ is contained only in either $\partial f$ or $\partial^+\! f$. 
To this end, by a diagonal subsequence argument we derive from \eqref{e:lim_eps_subdiff} that 
$\partial f$ and $\partial^+\! f$ both have the robustness property. 
Thus, in either case, the result follows.
\end{proof}

\begin{lemma}[Upper semicontinuity of the symmetric subdifferential]
\label{l:upsemicont}
Let $f\colon X \to [-\infty, +\infty]$ be Lipschitz continuous around $x \in X$ with $|f(x)| <+\infty$, 
and consider sequences $(x_n)_\nnn$ and $(x_n^*)_\nnn$ in $X$ such that $x_n \to x$ and $x_n^* \in \partial^0\! f(x_n)$ for every $\nnn$. 
Then $(x_n^*)_\nnn$ is bounded and its cluster points are contained in $\partial^0\! f(x)$.
\end{lemma}
\begin{proof}
By assumption, there exist a neighborhood $U$ of $x$ and a constant $\ell \in \RP$ such that
$f$ is Lipschitz continuous on $U$ with modulus $\ell$. 
In particular, $f$ and $-f$ are Lipschitz continuous around each $u \in U$ with modulus $\ell$. 
By \cite[Corollary~1.81]{Mor06} and \eqref{e:symdiff}, we have that
\begin{equation}
(\forall u \in U)(\forall u^* \in \partial^0\! f(u))\quad \|u^*\| \leq \ell.
\end{equation}
Since $x_n \to x$, it follows that $(x_n^*)_\nnn$ is bounded. 

Let $x^*$ be a cluster point of $(x_n^*)_\nnn$. 
Then there is a subsequence $(x_{k_n}^*)_\nnn$ converging to $x^*$.  
Noting that $x_{k_n} \to x$, the Lipschitz continuity of $f$ around $x$ yields $x_{k_n} \stackrel{f}{\to} x$. 
Now apply \cref{l:robustness}.
\end{proof}

\subsection{Coercivity}
Recall that a function $f\colon X \to \left[-\infty, +\infty\right]$ is \emph{coercive} if 
\begin{equation}
\lim_{\|x\|\to +\infty} f(x) =+\infty.
\end{equation}
For convenience, we recall some basic properties of coercivity.
\begin{fact}[Coercive functions]
\label{f:coer}
Let $f\colon X \to \left[-\infty, +\infty\right]$.
Then the following hold:
\begin{enumerate}
\item 
\label{f:coer_lev}
$f$ is coercive if and only if its lower level sets $\lev{\xi}f$ are bounded for all $\xi \in \RR$.
\item 
\label{f:coer_inf}
If $f$ is proper, convex, and coercive, then $\inf f(X) >-\infty$. 

\end{enumerate}
\end{fact}
\begin{proof}
\ref{f:coer_lev}: \cite[Proposition~11.11]{BC11}. 
\ref{f:coer_inf}: \cite[Lemma~2.13]{BCN06}. 
\end{proof}

The following preparatory lemma shows that coercivity is preserved under direct sums.
\begin{lemma}
\label{l:sumcoer}
Let $f\colon X \to \left[-\infty, +\infty\right]$ and let $g\colon Y \to \left[-\infty, +\infty\right]$, 
where $X$ and $Y$ are real Hilbert spaces.
Set $h\colon X\times Y \to \left[-\infty, +\infty\right]\colon (x, y)\mapsto f(x) +g(y)$.
Suppose that $f$ and $g$ are proper, convex, and coercive on $X$ and $Y$, respectively. 
Then $h$ is proper, convex, and coercive on $X\times Y$.
\end{lemma}
\begin{proof}
It immediately follows by assumption and definition that $h$ is proper and convex. 
Now \cref{f:coer}\ref{f:coer_inf} implies that 
\begin{equation}
\label{e:inf_fg}
\inf f(X) >-\infty \quad\text{and}\quad \inf g(Y) >-\infty.
\end{equation}
Suppose, by way of a contradiction, that $h$ is not coercive. Then, there exists a sequence $(x_n, y_n)_\nnn$ in $X\times Y$ such that $\|(x_n, y_n)\| =\sqrt{\|x_n\|^2 +\|y_n\|^2} \to +\infty$ and $(h(x_n, y_n))_\nnn$ is bounded above
that is, there exists $\mu\in\RR$ such that
\begin{equation}
(\forall\nnn)\quad h(x_n, y_n) =f(x_n) +g(y_n) \leq \mu. 
\end{equation} 
Combining with \eqref{e:inf_fg}, we obtain that $(f(x_n))_\nnn$ and $(g(y_n))_\nnn$ are bounded above.  
But since $f$ and $g$ are coercive, $(x_n)_\nnn$ and $(y_n)_\nnn$ must therefore be bounded, and thus so is $(x_n, y_n)_\nnn$
which contradicts the fact that $\|(x_n, y_n)\| \to +\infty$.
\end{proof}

\section{The Douglas--Rachford algorithm for finding a zero of a function}
\label{s:main}
From herein, we assume that
\begin{empheq}[box=\mybluebox]{equation}
\text{$f \colon X \to \left[-\infty, +\infty\right]$ is proper with closed graph.}
\end{empheq}
Note that, since $f$ is assumed proper, $\gra f$ is necessarily a closed set whenever $f$ is continuous throughout its effective domain in the sense that
\begin{equation}
(\forall x \in \dom f)\quad f(x) =\lim_{\substack{y \stackrel{\dom f}{\longrightarrow} x}} f(y).
\end{equation}
As the following examples show, the converse need not hold (\emph{i.e.,}~the graph of a discontinuous function can be closed) and, in general, mere lsc is not sufficent to ensure closedness of the graph.
\begin{example}[A discontinuous, lsc function with closed graph]
Consider $f\colon\RR\to\RR$ defined by
	   \begin{equation}
	    f(x) := \begin{cases}
	                1/|x| &\text{~if~} x\neq 0, \\
	                0   &\text{~if~} x=0.
	              \end{cases}
	  \end{equation}
Then $f$ is continuous except at $x=0$ where it is merely lsc. In particular, $f$ is lsc but not continuous. However, $f$ does have a closed graph. Indeed, the graph of $f$ may be expressed as the union of two closed sets: $\gra f=\gra\left(1/|\cdot|\right)\cup \{(0,0)\}$ where we note that $\gra\left(1/|\cdot|\right)$ is closed since $x\mapsto 1/|x|$ is continuous on its domain. \qed
\end{example}

It is known, see for instance \cite[Corollary~9.15]{BC11}, that every proper lsc convex function $f\colon \RR \to \left[-\infty, +\infty\right]$ 
is continuous throughout the closure of $\dom f$ and hence has a closed graph. 
However, this does not hold for proper lsc convex functions in $\RR^2$ which, as a consequence, gives rise to the following example.
\begin{example}[A proper lsc convex function with nonclosed graph]
Consider $f\colon \RR^2 \to \left[-\infty, +\infty\right]$ defined by
\begin{equation}
f(\alpha, \beta) :=\begin{cases}
\beta^2/\alpha &\text{~if~} \alpha >0, \\
0 &\text{~if~} (\alpha, \beta) =0, \\
+\infty &\text{~otherwise}.
\end{cases}
\end{equation}
Then $f$ is proper, lsc, and convex, as shown in \cite[Example~9.27]{BC11}.
Now setting $(\forall\nnn)$ $x_n =(1/(n +1)^2, 1/(n+1))$, 
we have that the sequence $(x_n, f(x_n))_\nnn$ lies in $\gra f$ but its limit $((0, 0), 1) \notin \gra f$, hence $\gra f$ is not closed.\qed
\end{example}

Our focus is the feasibility problem \eqref{e:prob} in the product Hilbert space $X\times \RR$ with constraints 
\begin{empheq}[box=\mybluebox]{equation}\label{eq:def A and B}
A :=X\times \{0\} \quad\text{and}\quad B :=\gra f, 
\end{empheq}
where $A\cap B \neq\varnothing$.
Note that, in the case in which $B$ is the \emph{epigraph} of a proper lower semicontinuous function (and hence a nonempty closed convex set), the convergence of the Douglas--Rachford algorithm was previously studied in \cite{BD16,BDNP16a,BDNP16}. Until this work, the case in which $B$ is the \emph{graph} of a proper function had not been considered even for the class of convex functions. It is also clear that, equivalently, our problem may be posed as
\begin{empheq}[box=\mybluebox]{equation}
\text{find a zero of the function $f$}
\end{empheq}
under the assumption that $f^{-1}(0) \neq\varnothing$.
In what follows, the sequence $(z_n)_\nnn$ shall denote a DR sequence for \eqref{eq:def A and B}, that is, any sequence which satisfies
\begin{empheq}[box=\mybluebox]{equation}
z_0 =(x_0, \rho_0) \in X\times \RR \quad\text{and}\quad (\forall\nnn)\quad z_{n+1} =(x_{n+1}, \rho_{n+1}) \in T_{A,B}z_n.
\end{empheq}
In this setting, the projector onto $A$ and the reflector across $A$ are given, respectively, by
\begin{equation}
\label{e:projA}
(\forall (x, \rho) \in X\times \RR)\quad P_A(x, \rho) =(x, 0),
\quad\text{and}\quad R_A(x, \rho) =(x, -\rho).
\end{equation}
Although the two possible DR operators, $T_{A,B}$ and $T_{B,A}$, associated with $A$ and $B$ give different algorithms, 
since $A$ is a subspace, it holds that $T_{B,A}^n =R_AT_{A,B}^nR_A$ for every $n\in\mathbb{N}$
(see \cite[Theorem~2.7(i) \& Remark~2.10(ii)-(iii)]{BM16}). Thus in order to study the DRA corresponding to $T_{B,A}$ it suffices just to study the DRA corresponding to $T_{A,B}$.

To begin, we collect some preparatory lemmas which we use to give a precise description of the DR iteration for the sets $A$ and $B$ in \eqref{eq:def A and B}. Our first result is concerned with the range of the DR operator.
\begin{lemma}[Range of $T_{A,B}$]
\label{l:ranT}
The following assertions hold.
\begin{enumerate}
\item 
\label{l:ranT_eq}
$(\forall (x, \rho) \in X\times \RR)\quad T_{A,B}(x, \rho) =(0, \rho) +P_B(x, -\rho)$.
\item 
\label{l:ranT_incl}
$\ran T_{A,B} :=T_{A,B}(X\times \RR) \subseteq \dom f \times \RR$.
\end{enumerate}
\end{lemma}
\begin{proof}
Let $(x, \rho) \in X\times \RR$. \ref{l:ranT_eq}: Combining \cref{l:1cvx}\ref{l:1cvx_single} and \eqref{e:projA} yields 
\begin{align}
T_{A,B}(x, \rho) &=(\Id -P_A +P_BR_A)(x, \rho) 
=(x, \rho) -(x, 0) +P_B(x, -\rho) 
=(0, \rho) +P_B(x, -\rho).
\end{align}
\ref{l:ranT_incl}: Since $B \subseteq \dom f \times \RR$, it follows from \ref{l:ranT_eq} that 
\begin{equation}
T_{A,B}(x, \rho) \subseteq (0, \rho) +B \subseteq (0, \rho) +\dom f \times \RR \subseteq \dom f \times \RR,
\end{equation}
which completes the proof.
\end{proof}

Note that, in view of \cref{l:ranT}\ref{l:ranT_incl}, from now on it suffices to assume that
\begin{empheq}[box=\mybluebox]{equation}
(\forall\nnn)\quad z_n =(x_n, \rho_n) \in \dom f \times \RR.
\end{empheq}
In the following lemma, we turn our attention to the projector onto $B=\gra f$. The provided characterization for $P_B$ will then be used in Lemma~\ref{l:DRstep} to describe the DR operator relative to $(A,B)$.
\begin{lemma}[Projector onto the graph of $f$]
\label{l:gra}
Let $(x, \rho) \in \dom f \times \RR$. Then $P_B(x, \rho)\neq\varnothing$ and, for any $(p, \pi) \in P_B(x, \rho)$, it holds that $p \in \dom f$ and $\pi =f(p)$. In addition, the following assertions hold.
\begin{enumerate}
\item
\label{l:gra_Lipschitz}
If $f$ is Lipschitz continuous around $p$, then 
\begin{subequations}
\label{e:proj_gra'}
\begin{align}
x &\in \begin{cases}
p +(f(p) -\rho)\partial f(p) &\text{~if~} f(p) \geq \rho, \\ 
p +(f(p) -\rho)\partial^+\! f(p) &\text{~if~} f(p) < \rho  
\end{cases} \\ 
&\subseteq p +(f(p) -\rho)\partial^0\! f(p).  
\end{align}
\end{subequations}
\item 
\label{l:gra_convex}
If $f$ is convex and $p \in \intdom f$, then 
\begin{equation}
\label{e:proj_gra}
x \in p +(f(p) -\rho)\partial f(p).
\end{equation}
\end{enumerate}
\end{lemma}
\begin{proof}
The existence of a point $(p, \pi) \in P_B(x, \rho)$ is ensured since the set $B =\gra f$ is a nonempty closed subset of $X\times \RR$. 
Since $(p, \pi) \in B =\gra f$, it holds that $p \in \dom f$ and $\pi =f(p)$.

\ref{l:gra_Lipschitz}:~Since $(p, f(p)) \in P_B(x, \rho)$, we have that
\begin{equation}
p \in \argmin_{y \in X} \|y -x\|^2 +|f(y) -\rho|^2,
\end{equation}
and, applying Fermat's rule (\cref{f:Fermat}), gives
\begin{equation}
0 \in \partial\big( \|\cdot -x\|^2 +(f(\cdot) -\rho)^2 \big)(p).
\end{equation} 
Using the sum and product rules (\cref{f:calculus}) and
noting that $\|\cdot -x\|^2$ is continuously (Fr\'echet) differentiable and hence strictly differentiable on $X$ with $\nabla\|\cdot -x\|^2(p) =2(p -x)$ (see, for instance, \cite[Example~16.11~\&~Corollary~17.36]{BC11}), 
we deduce that
\begin{equation}
0 \in \nabla\|\cdot -x\|^2(p) +\partial(f(\cdot) -\rho)^2(p) \\
=2(p -x) +\partial\big( 2(f(p) -\rho)(f(\cdot) -\rho) \big)(p). 
\end{equation}
Now by the scalar multiplication rule (\cref{l:scalar}), 
\begin{equation}
\partial\big( 2(f(p) -\rho)(f(\cdot) -\rho) \big)(p) 
=\begin{cases}
2(f(p) -\rho)\partial f(p) &\text{~if~} f(p) > \rho, \\ 
2(f(p) -\rho)\partial^+\! f(p) &\text{~if~} f(p) < \rho.
\end{cases} 
\end{equation}
Finally, if $f(p) =\rho$, then, since by assumption $p\in\inte\dom f$, the function $2(f(p) -\rho)(f(\cdot) -\rho)$ is zero around $p$. Consequently, $\partial\big( 2(f(p) -\rho)(f(\cdot) -\rho) \big)(p) =\{0\} =2(f(p) -\rho)\partial f(p)$ where $\partial f(p) \neq\varnothing$ due to \cite[Corollary~2.25]{Mor06}. Altogether, we have proven \eqref{e:proj_gra'}.

\ref{l:gra_convex}:~Since $f$ is proper and convex,  $f$ is locally Lipschitz continuous on $\intdom f$ \cite[Corollary~8.32]{BC11}. The claim thus follows from \ref{l:gra_Lipschitz}.
\end{proof}

\begin{lemma}[One DR step]
\label{l:DRstep}
Let $(x, \rho) \in \dom f \times \RR$ 
and let $(x_+, \rho_+) \in T_{A,B}(x, \rho)$. 
Then 
\begin{subequations}
\label{e:DRstep0}
\begin{gather}
(x_+, f(x_+)) \in P_B(x, -\rho), \quad \rho_+ =\rho +f(x_+), \quad\text{and} \\
\|x -x_+\|^2 \leq (f(x) -f(x_+))(f(x) +f(x_+) +2\rho).
\end{gather}
\end{subequations}
Suppose, in addition, that $f$ is Lipschitz continuous around $x_+$.
Then there exists $x^* \in \partial^0\! f(x_+)$ such that
\begin{equation}
\label{e:DRstep}
\rho_+ =\rho +f(x_+), \quad x_+ =x -\rho_+x_+^*, 
\quad\text{and}\quad x_+^* \in \begin{cases}
\partial f(x_+) &\text{~if~} \rho_+ \geq 0, \\ 
\partial^+\! f(x_+)  &\text{~if~} \rho_+ <0.
\end{cases} 
\end{equation}
and, furthermore, the following assertions hold.
\begin{enumerate}
	\item\label{l:DRstep_i} If $f$ is strictly differentiable at $x_+$ with $\nabla f(x_+) =0$, then $x_+ =x$.
	\item\label{l:DRstep_ii} If $f$ is convex and $0\in\partial f(x_+)$, then either $x_+=x$ or $0\not\in\partial f(x)$.
\end{enumerate}	
\end{lemma}
\begin{proof}
It follows from \cref{l:ranT}\ref{l:ranT_eq} that 
$(x_+, \rho_+ -\rho) =(x_+, \rho_+) -(0, \rho) \in P_B(x, -\rho)$ 
and from \cref{l:gra} that $x_+ \in \dom f$ and $\rho_+ -\rho =f(x_+)$.
Altogether, $(x_+, f(x_+)) \in P_B(x, -\rho)$ and $\rho_+ =\rho +f(x_+)$.
The former implies that
\begin{equation}
\|x_+ -x\|^2 +|f(x_+) +\rho|^2 \leq \|x -x\|^2 +|f(x) +\rho|^2 =|f(x) +\rho|^2,
\end{equation}
which completes the proof of \eqref{e:DRstep0}.

Now assume that $f$ is Lipschitz continuous around $x_+$.
By \cref{l:gra}\ref{l:gra_Lipschitz},  
\begin{equation}
x =x_+ +(\rho +f(x_+))x_+^* 
\quad\text{for some~} x_+^* \in \begin{cases}
\partial f(x_+) &\text{~if~} f(x_+) \geq -\rho, \\ 
\partial^+\! f(x_+)  &\text{~if~} f(x_+) < -\rho,
\end{cases} 
\end{equation} 
from which \eqref{e:DRstep} follows since $\partial^0\! f(x_+) =\partial f(x_+)\cup \partial^+\! f(x_+)$. Furthermore, we argue as follows.

\ref{l:DRstep_i}: If $f$ is strictly differentiable at $x_+$ with $\nabla f(x_+) =0$, 
then $\partial f(x_+) =\partial^+\! f(x_+) =\partial^0\! f(x_+) =\{0\}$, and so $x_+^* =0$, which gives $x_+ =x$.

\ref{l:DRstep_ii}: Suppose $f$ is convex, $0\in\partial f(x_+)$ and $x_+\neq x$. 
Then \eqref{e:DRstep0} yields
\begin{equation}
\label{e:DRstep0'}
0 <\|x -x_+\|^2 \leq (f(x) -f(x_+))(f(x) +f(x_+) +2\rho).
\end{equation}
Since $0\in\partial f(x_+)$, we have $f(x_+)=\min f(X)$ and hence $f(x_+) \leq f(x)$. 
By \eqref{e:DRstep0'}, the inequality is actually strict, that is, $f(x_+) < f(x)$ which implies that $f(x) >\min f(X)$ and hence $0\not\in\partial f(x)$.
\end{proof}

Recall that the set of \emph{fixed points} of $T_{A,B}$ is the set $\Fix T_{A,B} :=\menge{z \in X\times \RR}{z \in T_{A,B}z}$. 
If $A$ and $B$ were convex sets, the fixed point of the DR operator can be precisely described \cite[Corollary~3.9]{BCL04}. 
Although $B$ is not convex in our setting, we are still, nevertheless, able to arrive at the following satisfactory characterization.

\begin{lemma}[Fixed points of $T_{A,B}$]
\label{l:FixT}
The following assertions hold.
\begin{enumerate}
\item 
\label{l:FixT_basic}
$(\forall (x, \rho) \in \Fix T_{A,B})$ $f(x) =0$ and $(x, 0) \in P_B(x, -\rho)$. Consequently,
\begin{equation}
A\cap B\subseteq \Fix T_{A,B}\subseteq f^{-1}(0)\times \RR
\quad\text{and}\quad P_A\Fix T_{A,B} =A\cap B.
\end{equation}
\item
\label{l:FixT_positive}
If $\min f(X) =0$, then
\begin{equation}
A\cap B \subseteq f^{-1}(0)\times \RP \subseteq \Fix T_{A,B}.
\end{equation}

\item
\label{l:FixT_Lipschitz}
If $f$ is locally Lipschitz continuous on $f^{-1}(0)$, then 
\begin{equation}
\Fix T_{A,B}\subseteq (A\cap B)\cup 
\left( f^{-1}(0)\cap (\partial f)^{-1}(0)\times \RPP \right)\cup \left( f^{-1}(0)\cap (\partial^+\! f)^{-1}(0)\times \RMM \right).
\end{equation}
In particular, if $f^{-1}(0)\cap (\partial f)^{-1}(0) =f^{-1}(0)\cap (\partial^+\! f)^{-1}(0) =\varnothing$, 
then $\Fix T_{A,B} =A\cap B$.
\item 
\label{l:FixT_convex}
If $f$ is convex and $f^{-1}(0) \subseteq \intdom f$, then 
\begin{equation}\label{eq:FixT_convex}
\Fix T_{A,B}\subseteq (A\cap B)\cup \left( f^{-1}(0)\cap (\partial f)^{-1}(0)\times \RR\smallsetminus \{0\}\right).
\end{equation} 
In particular, if $\inf f(X) <0$, then $\Fix T_{A,B} =A\cap B$.
\end{enumerate}	
\end{lemma}
\begin{proof}
\ref{l:FixT_basic}:~Let $(x, \rho) \in \Fix T_{A,B}$. Then, by \cref{l:ranT}\ref{l:ranT_eq}, we have
\begin{equation}
\label{e:(x,0)}
(x, \rho) \in T_{A,B}(x, \rho) =(0, \rho) +P_B(x, -\rho) \iff (x, 0) \in P_B(x, -\rho).
\end{equation}
On the one hand, \eqref{e:(x,0)} implies $(x, 0) \in B =\gra f$, so that $f(x) =0$, and hence $(x, \rho) \in f^{-1}(0) \times \RR$. 
On the other hand, \eqref{e:(x,0)} gives
\begin{equation}
P_A(x, \rho) =(x,0) \in P_B(x, -\rho),
\end{equation}
which proves that $P_A(x, \rho) \in A\cap B$. 
We deduce that $\Fix T_{A,B}\subseteq f^{-1}(0) \times \RR$ and $P_A\Fix T_{A,B} \subseteq A\cap B$. 
It straight-forward to show that $A\cap B \subseteq \Fix T_{A,B}$ 
from which it follows that $A\cap B =P_A(A\cap B) \subseteq P_A\Fix T_{A,B}$. 

\ref{l:FixT_positive}:~We immediately have that $A\cap B =f^{-1}\times \{0\} \subseteq f^{-1}(0)\times \RP$.
Now let $(x, \rho) \in f^{-1}(0)\times \RP$. Again by \cref{l:ranT}\ref{l:ranT_eq}, $T_{A,B}(x, \rho) =(0, \rho) +P_B(x, -\rho)$. 
It follows from $\min f(X) =0 =f(x)$ and $\rho \in \RP$ that 
\begin{equation}
\argmin_{y \in X} \|y -x\|^2 +|f(y) +\rho|^2 =x
\end{equation} 
and therefore $P_B(x, -\rho) =(x, f(x)) =(x, 0)$, which yields $T_{A,B}(x, \rho) =(0, \rho) +(x, 0) =(x, \rho)$, that is, $(x, \rho) \in \Fix T_{A,B}$. Hence $f^{-1}(0)\times \RP \subseteq \Fix T_{A,B}$.

\ref{l:FixT_Lipschitz}:~Let $(x, \rho) \in \Fix T_{A,B}$. 
By \ref{l:FixT_basic}, $f(x) =0$ and $(x, 0) \in P_B(x, -\rho)$. 
If $\rho =0$, then $f(x) =\rho =0$, and hence the fixed point $(x, \rho) \in A\cap B$. 
If $\rho \neq 0$, then, by using \cref{l:gra}\ref{l:gra_Lipschitz},  
\begin{equation}
x \in \begin{cases}
x +(0 +\rho)\partial f(x) & \text{if }\rho>0,  \\
x +(0 +\rho)\partial f^+(x) &\text{if }\rho<0
\end{cases}
\quad\implies\quad
0 \in \begin{cases}
\partial f(x) & \text{if }\rho>0,   \\
\partial f^+(x) &\text{if }\rho<0. \\
\end{cases}	      
\end{equation} 
Thus either $(x, \rho) \in f^{-1}(0)\cap (\partial f)^{-1}(0)\times \RPP$ 
or $(x, \rho) \in f^{-1}(0)\cap (\partial^+\! f)^{-1}(0)\times \RMM$ 
which completes the proof of the claim.

\ref{l:FixT_convex}:~By the assumptions on $f$ and \cite[Corollary~8.32]{BC11}, 
$f$ is locally Lipschitz continuous on $f^{-1}(0) \subseteq \intdom f$. 
The first claim by applying \ref{l:FixT_Lipschitz} 
and notating from convexity that $\partial f =\partial f^0 =\partial f\cup \partial^+\! f$ (see \eqref{e:cvx_inclus}).

To prove the second claim, suppose that there exists $x \in f^{-1}(0)\cap (\partial f)^{-1}(0)$, that is,
$f(x) =0$ and $0 \in \partial f(x)$. But then $\min f(X) =f(x) =0$ which contradicts the assumption that $\inf f(X) <0$, hence we deduce that $f^{-1}(0)\cap (\partial f)^{-1}(0) =\varnothing$. The conclusion follows.
\end{proof}

Roughly speaking, Lemma~\ref{l:FixT} shows that the fixed point set of $T_{A,B}$ consists of two parts: the intersection $A\cap B$ and a set containing critical points of $f$. In  the following result, we give conditions under which the DRA stays away from critical points.

For convenience, we denote $\Delta:=T_{A,B}(\dom f \times \RR)$ and the \emph{first coordinate projection}  by
\begin{equation}
\Pi\colon X \times \RR \to X\colon (y, \sigma) \mapsto y.
\end{equation}
\begin{corollary}
\label{c:bounded}
Suppose that one of the following holds:
\begin{enumerate}
\item
\label{c:bounded_Lipschitz}
$f$ is locally Lipschitz continuous on $\Pi(\Delta)$
and, for every $x \in (\nabla f)^{-1}(0)$, either 
\begin{equation}
\label{e:bounded}
f(x) <\min\{0, \sup f(\dom f)\}\quad\text{or}\quad f(x) >\max\{0, \inf f(X)\}.
\end{equation}
\item 
\label{c:bounded_convex}
$f$ is convex with $\dom f$ open, and $\inf f(X) <\min\{0, \sup f(\dom f)\}$. 
\end{enumerate}
Then the set $S :=\menge{\nnn}{\text{$f$ is strictly differentiable at $x_n$ with $\nabla f(x_n) =0$}}$ is bounded.
\end{corollary}
\begin{proof}
\ref{c:bounded_Lipschitz}: By way of a contraction, suppose that $S$ is unbounded. 
In this case, we claim that $S =\mathbb{N}$ and that the sequence $(x_n)_{\nnn}$ is constant. 
To see this, observe that if $n \in S$ (\emph{i.e.,} $f$ is strictly differentiable at $x_n$ with $\nabla f(x_n) =0$), 
then \cref{l:DRstep}\ref{l:DRstep_i} yields that $x_{n-1} =x_n$. 
In particular, $f$ is strictly differentiable at $x_{n-1}$ with $\nabla f(x_{n-1}) =0$. 
The claim now follows by descending induction on $n$.
	
Now, set $x :=x_0 =x_n$ for any $\nnn$. Let $y \in \dom f$. 
For all $\nnn$, since $(x_{n+1}, f(x_{n+1})) \in P_B(x_n, -\rho_n)$, the definition of $P_B$ implies
\begin{subequations}
\begin{align}
\|x_{n+1} -x_n\|^2 +|f(x_{n+1}) +\rho_n|^2 &\leq \|y -x_n\|^2 +|f(y) +\rho_n|^2 \\
\iff(f(x) -f(y))(f(x) +f(y) +2\rho_n) &\leq \|y -x\|^2. \label{e:fxfy}
\end{align}
\end{subequations}
Since $\nabla f(x)=0$, \eqref{e:bounded} implies that either $f(x)<0$ or $f(x)>0$. In the former case, \cref{l:DRstep} implies $\rho_{n+1}=\rho_0+nf(x)\to-\infty$ as $n\to\infty$, and hence $f(x) +f(y) +2\rho_n\to-\infty$. Since $\|y -x\|^2$ is fixed, \eqref{e:fxfy} implies that $f(x)-f(y)\geq0$. Since  $y \in \dom f$ was chosen arbitrarily, $f(x) =\sup f(\dom f)$, which contradicts the fact that $f(x) =\inf f(X) <\sup f(\dom f)$. The case in which $f(x)>0$ is proven analogously.

\ref{c:bounded_convex}: By assumption and \cite[Corollary~8.32]{BC11}, $f$ is locally Lipschitz continuous on $\dom f\supseteq \Pi(\Delta)$. By convexity of $f$, if $x\in (\nabla f)^{-1}(0)$, then $f(x) =\inf f(X) <\min\{0, \sup f(\dom f)\}$, hence \eqref{e:bounded} is satisfied. The result now follows from \ref{c:bounded_Lipschitz}.
\end{proof}

\begin{remark}
A convex function is strictly differentiable at every point where it is G\^ateaux differentiable. Indeed, supposing that a function $f$ is convex and G\^ateaux differentiable at $x \in \dom f$, it then follows, from \eqref{e:Fensubdiff} and \cite[Proposition~17.26(i)]{BC11}, that $\partial f(x) =\{\nabla f(x)\}$ is a singleton, 
and, from \cite[Proposition~17.39]{BC11}, that $f$ is lsc at $x$ and $x \in \intdom f$.
By combining with \cite[Proposition~8.12~\&~Theorem~9.18(a)~\&~(c)]{RW98}, $f$ is strictly differentiable at $x$. 
\end{remark}

The following result shows that, under a differentiation assumption, the inverse of the DR operator is continuous. This property, and its connection to stability, is explored further in Section~\ref{s:stability of DRA}.
\begin{corollary}
\label{c:inverseT}
Suppose that $f$ is strictly differentiable on an open set $U$ contained in $\Pi(\Delta)$. Then 
\begin{equation}
\label{e:inverseT}
(\forall (y, \sigma) \in \Pi^{-1}(U))\quad T_{A,B}^{-1}(y, \sigma) =(y +\sigma\nabla f(y),  \sigma -f(y)),
\end{equation}
and $T_{A,B}^{-1}$ is continuous on $\Pi^{-1}(U)$. Consequently, if the limit of a convergent DR sequence is contained in $\Pi^{-1}(U)$, then it is necessarily a fixed point $z$ of $T_{A,B}$ with $P_Az \in A\cap B$.
\end{corollary}
\begin{proof}
Let $(y, \sigma) \in \Pi^{-1}(U)$. 
Then $y \in U$ and there exists $(x, \rho) \in \dom f \times \RR$ such that $(y, \sigma) \in T_{A,B}(x, \rho)$.
Since $f$ is strictly differentiable on $U$, it is Lipschitz continuous around $y$ with 
\begin{equation}
	\partial f(y) =\partial^+\! f(y) =\{\nabla f(y)\}.
\end{equation}
By \cref{l:DRstep}, $\sigma =\rho +f(y)$ and $y =x -\sigma\nabla f(y)$, which proves \eqref{e:inverseT}. 
To deduce the continuity of $T^{-1}_{A,B}$, observe that, since $f$ is strictly differentiable on $U$, $\nabla f$ is continuous on $U$ \cite[Corollary~9.19(a)--(b)]{RW98}.

Finally, let $(z_n)_\nnn$ be a DR sequence which converges to a point $z =(x, \rho) \in U$. 
Without loss of generally, we can and do assume that $z_n =(x_n, \rho_n) \in U$ for every $\nnn$. 
Then, using the continuity of $T^{-1}_{A,B}$ and the fact that $z_{n-1} =T^{-1}_{A,B}(z_n)$ gives
  \begin{equation} 
  	T^{-1}_{A,B}(z) = \lim_{n\to\infty} T^{-1}_{A,B}(z_n) = \lim_{n\to\infty}z_{n-1} = z,
  \end{equation}
which shows that $z\in T_{A,B}z$ and thus $z \in \Fix T_{A,B}$. 
In turn, applying \cref{l:FixT}\ref{l:FixT_basic} yields $P_Az \in P_A\Fix T_{A,B} =A\cap B$.
\end{proof}

\section{Stability and local convergence}
\label{s:stability of DRA}
In this section, we use an inverse function argument to give a condition under which the DRA algorithm is stable around fixed points in the sense of Lipschitz continuity. Again, we emphasize that alone such results do not guarantee convergence of the DRA. This question will be addressed in Section~\ref{s:lyapunov}. 

To begin, we recall two facts which will be of use: an inverse function theorem, and the so-called \emph{Sherman--Morrison formula}.
\begin{fact}[Single-valued Lipschitzian invertibility]
\label{f:inverse}
Let $T\colon \RR^m \to \RR^m$ be strictly differentiable at $\bar{x}$. If $\nabla T(\bar{x})$ is nonsingular, then $T^{-1}$ has a Lipschitz continuous single-valued localization, $S$, around $\bar{y} :=T\bar{x}$ for $\bar{x}$. 
Moreover, the Lipschitz modulus of $S$ at $\bar{y}$ is equal to $\|\nabla T(\bar{x})^{-1}\|$ 
and $S$ is strictly differentiable at $\bar{y}$ with $\nabla S(\bar{y}) =\nabla T(\bar{x})^{-1}$.
\end{fact}	   
\begin{proof}
 This is a special case of \cite[Corollary~9.55]{RW98}.
\end{proof}
\begin{fact}[Sherman--Morrison formula]
\label{f:SMformula}
Let $M$ be a nonsingular square matrix and let $u$ and $v$ be column vectors of appropriate dimensions so that the following multiplication operators are well defined. Then the following assertions hold.
\begin{enumerate}
\item 
\label{f:SMformula_ori}
If $1 +v^\top M^{-1}u \neq 0$, then $M +uv^\top$ is nonsingular and
\begin{equation}
(M +uv^\top)^{-1} =M^{-1} -\frac{1}{1 +v^\top M^{-1}u} M^{-1}uv^\top M^{-1}.
\end{equation}
\item
\label{f:SMformula_inv}
If $M +uv^\top$ is singular, then $1 +v^\top M^{-1}u =0$.
\end{enumerate}
\end{fact}
\begin{proof}
\ref{f:SMformula_ori}: See \cite{SM50}. \ref{f:SMformula_inv}: This is the contrapositive of  \ref{f:SMformula_ori}.
\end{proof}	

We are ready to prove our main result regarding stability of the DRA. In the following,  $\succeq$ denotes the \emph{L\"owner partial order} on the space of symmetric matrices.
We say that $f$ is \emph{twice strictly differentiable} at $\bar{x}$ if $f$ is differentiable around $\bar{x}$ and $\nabla f$ is strictly differentiable at $\bar{x}$.

\begin{theorem}[Stability of the DRA]
\label{t:stable}
Let $\bar{z} :=(\bar{x},\bar{\rho}) \in \Fix T_{A,B}$, and suppose that $f$ is twice strictly differentiable at $\bar{x}$ and that $\bar{\rho}\nabla^2f(\bar{x})\succeq0$. 
Then $(T_{A,B}^{-1})^{-1}$ has a Lipschitz continuous single-valued localization, $S$, around $\bar{z}$ for $\bar{z}$ which is strictly differentiable at $\bar{z}$ and has Lipschitz modulus at $\bar{z}$ equal to $\ell \leq 1$ where
\begin{equation}
\label{e:Lmodulus}
\ell=\|\nabla S(\bar{z})\| =\begin{cases}
1 &\text{~if~} \dim X >1, \\ 
\frac{1}{\sqrt{1 +|f'(\bar{x})|^2}} &\text{~if~} \dim X =1. 
\end{cases}
\end{equation}
Furthermore, if $\bar{z}=(\bar{x},0)\in A\cap B\subseteq\Fix T_{A,B}$, then $S$ and $T_{A,B}$ coincide on a neighborhood of $\bar{z}$.
\end{theorem}
\begin{proof}
Since $f$ is twice strictly differentiable at $\bar{x}$, $\nabla f$ both exists and is Lipschitz continuous around $\bar{x}$. 
In particular, $f$ is continuous differentiable around $\bar{x}$ and, consequently, strictly differentiable around $\bar{x}$. 
Therefore, for every $(x_+, \rho_+) \in X\times \RR$ with $x^+$ sufficiently close to $\bar{x}$, \cref{c:inverseT} gives that
\begin{equation}
T^{-1}_{A,B}\begin{bmatrix} x_+\\ \rho_+ \\ \end{bmatrix}
= \begin{bmatrix} x_+ +\rho_+\nabla f(x_+)\\ \rho_+-f(x_+) \\ \end{bmatrix} 
\end{equation}
and, since $\nabla f$ is strictly differentiable at $\bar{x}$, $T^{-1}_{A,B}$ is strictly differentiable at $\bar{z}$ with Jacobian given by
\begin{equation}
\label{eq:T-1 jacobian}
\nabla T^{-1}_{A,B}(\bar{z}) =\nabla T^{-1}_{A,B}\begin{bmatrix} \bar{x} \\ \bar{\rho} \\ \end{bmatrix}
=\begin{bmatrix} \Id+\bar{\rho}\nabla^2f(\bar{x})  & \nabla f(\bar{x}) \\ -\nabla f(\bar{x})^\top & 1 \\ \end{bmatrix}.
\end{equation}
Now, by distinguishing two cases, we show that $\nabla T^{-1}(\bar{z})$ is nonsingular and
\begin{equation}
\label{e:Lmodulus'}
\|(\nabla T^{-1}_{A,B}(\bar z))^{-1}\| =\begin{cases}
1 &\text{~if~} \dim X >1, \\ 
\frac{1}{\sqrt{1 +|f'(\bar x)|^2}} &\text{~if~} \dim X =1. 
\end{cases}
\end{equation}

\emph{Case 1:} Assume $\bar\rho =0$. 
Then \eqref{eq:T-1 jacobian} becomes
\begin{equation}
\nabla T^{-1}_{A,B}(\bar z) =\begin{bmatrix} \Id & \nabla f(\bar x) \\ -\nabla f(\bar x)^\top & 1 \\ \end{bmatrix}
\end{equation}
and hence $\det \nabla T^{-1}_{A,B}(\bar{z}) =\det(\Id + \nabla f(\bar{x})\nabla f(\bar{x})^\top)$. 
Noting that
\begin{equation}
1 +\nabla f(\bar{x})^\top \Id \nabla f(\bar{x})  =1 +\|\nabla f(\bar{x})\|^2 \neq 0,
\end{equation}
it follows from \cref{f:SMformula}\ref{f:SMformula_ori} that $\Id + \nabla f(\bar{x})\nabla f(\bar{x})^\top$ is nonsingular and that
\begin{equation}
\label{e:invmatrix}
\left(\Id + \nabla f(\bar{x})\nabla f(\bar{x})^\top\right)^{-1} =\Id -\alpha\nabla f(\bar{x})\nabla f(\bar{x})^\top \text{ where }\alpha :=(1+\|\nabla f(\bar{x})\|^2)^{-1}.
\end{equation}
Therefore, $\det \nabla T^{-1}_{A,B}(\bar{z}) =\det(\Id + \nabla f(\bar{x})\nabla f(\bar{x})^\top) \neq 0$ and
hence, in particular, $\nabla T^{-1}_{A,B}(\bar{z})$ is nonsingular. 
To estimate $\|M\|$ where $M :=(\nabla T^{-1}_{A,B}(\bar{z}))^{-1}$, recall that 
$\|M\| =\sqrt{\lambda_{\max}(M^\top M)},$
where $\lambda_{\max}$ denotes the largest eigenvalue.
Using block matrix inversion and \eqref{e:invmatrix} gives 
\begin{subequations}
\begin{align}
M^\top M &=\big(\nabla T^{-1}_{A,B}(\bar{z})^\top\big)^{-1} (\nabla T^{-1}_{A,B}(\bar{z}))^{-1} 
=\left(\nabla T^{-1}_{A,B}(\bar{z})\nabla T^{-1}_{A,B}(\bar{z})^\top\right)^{-1} \\
&= \begin{bmatrix} \Id +\nabla f(\bar{x})\nabla f(\bar{x})^\top  & 0 \\ 0 & 1 +\|\nabla f(\bar{x})\|^2 \end{bmatrix}^{-1} \\
&=\begin{bmatrix} (\Id+\nabla f(\bar{x})\nabla f(\bar{x})^\top)^{-1}  & 0 \\ 0 & (1 +\|\nabla f(\bar{x})\|^2)^{-1} \end{bmatrix} \\
&=\begin{bmatrix} \Id -\alpha\nabla f(\bar{x})\nabla f(\bar{x})^\top & 0 \\ 0 & \alpha \end{bmatrix}
\end{align}
\end{subequations}
and so 
\begin{equation}
\lambda_{\max}(M^\top M) =\max\left\{\lambda_{\max}\big(\Id -\alpha\nabla f(\bar{x})\nabla f(\bar{x})^\top\big), \alpha\right\}.
\end{equation}
Let $\lambda$ be an eigenvalue of $\Id -\alpha\nabla f(\bar{x})\nabla f(\bar{x})^\top$, that is,
\begin{equation}
\label{e:det}
\det\left((1 -\lambda)\Id -\alpha\nabla f(\bar{x})\nabla f(\bar{x})^\top\right) =0.
\end{equation} 
If $\lambda =1$, then we must have $\det(-\alpha\nabla f(\bar{x})\nabla f(\bar{x})^\top) =0$, 
which occurs if and only if $\dim X >1$ or $\nabla f(\bar{x}) =0$.
Otherwise, using \eqref{e:det} and \cref{f:SMformula}\ref{f:SMformula_inv} yields
\begin{equation}
1 +\nabla f(\bar{x})^\top\frac{1}{1 -\lambda}\Id(-\alpha\nabla f(\bar{x})) =0
\implies 1 -\frac{1}{1 -\lambda}\alpha\left(\frac{1}{\alpha} -1\right) =0 
\implies\lambda=\alpha.
\end{equation}
Hence, either $\lambda =1$ or $\lambda =\alpha$. In either case, 
\begin{equation}
\lambda_{\max}(M^\top M) =\begin{cases}
1 &\text{~if~} \dim X >1 \text{~or~} \nabla f(\bar{x}) =0, \\ 
\alpha &\text{~otherwise},
\end{cases}
\end{equation}
and, by noting that $\alpha \leq 1$ with equality if and only if $\nabla f(\bar{x}) =0$, we deduce that
\begin{equation}
\label{eq:dimesion bound}
\|(\nabla T^{-1}_{A,B}(\bar{z}))^{-1}\| =\|M\| =\sqrt{\lambda_{\max}(M^\top M)} 
=\begin{cases}
1 &\text{~if~} \dim X >1, \\ 
\sqrt{\alpha} &\text{~if~} \dim X =1.
\end{cases}
\end{equation}

\emph{Case 2:} Assume $\bar\rho \neq 0$. 
Then $\bar z \in \Fix T_{A,B}\smallsetminus (A\cap B)$ and, by \cref{l:FixT}\ref{l:FixT_Lipschitz}, 
$\bar x \in f^{-1}(0)\cap (\nabla f)^{-1}(0)$ (\emph{i.e.,}~$f(\bar x) =0$ and $\nabla f(\bar x) =0$).
In turn, \eqref{eq:T-1 jacobian} becomes
\begin{equation}
\nabla T^{-1}_{A,B}(\bar z) =\begin{bmatrix} \Id +\bar\rho\nabla^2 f(\bar x) & 0 \\ 0 & 1 \\ \end{bmatrix},
\end{equation}
and, since  $\bar\rho\nabla f^2(\bar x)\succeq 0$ by assumption, we have $\Id +\bar\rho\nabla f^2(\bar x) \succeq \Id$ so that
\begin{equation}
\label{e:eigenvalue}
\lambda_{\min}\big(\Id +\bar\rho\nabla f^2(\bar x)\big) \geq 1 >0,
\end{equation}
where $\lambda_{\min}$ denotes the smallest eigenvalue.
We therefore have that both $\Id +\bar\rho\nabla f^2(\bar x)$ and $\nabla T^{-1}_{A,B}(\bar z)$ are nonsingular and, moreover, that

\begin{equation}
(\nabla T^{-1}_{A,B}(\bar z))^{-1} =\begin{bmatrix} (\Id +\bar\rho\nabla^2 f(\bar x))^{-1} & 0 \\ 0 & 1 \\ \end{bmatrix}.
\end{equation}
Using \eqref{e:eigenvalue} yields $0 < \lambda \leq 1$ for every eigenvalue $\lambda$ of $(\Id +\bar\rho\nabla^2 f(\bar x))^{-1}$, and as the matrix is symmetric, we have
\begin{equation}
\|(\nabla T^{-1}_{A,B}(\bar z))^{-1}\| =\max\left\{\big| \lambda_{\max}\big((\Id +\bar\rho\nabla^2 f(\bar x))^{-1}\big) \big|, 1\right\} =1.
\end{equation}
Noting that $\nabla f(\bar x) =0$, we see that this completes the proof of \eqref{e:Lmodulus'}.

In either of the above cases, we have that $\nabla T^{-1}_{A,B}$ is nonsingular at $\bar{z}$ and that $\|(\nabla T^{-1}_{A,B}(\bar{z}))^{-1}\|$ satisfies \eqref{e:Lmodulus'}. 
Now, as $\bar{z} \in T_{A,B}\bar{z}$ and $T_{A,B}^{-1}$ is single-valued at $\bar{z}$, it follows that $\bar{z}=T^{-1}_{A,B}\bar{z}$. 
By applying \cref{f:inverse}, we deduce that $(T^{-1}_{A,B})^{-1}$ has a Lipschitz continuous single-valued localization, $S$, around $\bar{z}$ for $\bar{z}$ which is strictly differentiable at $\bar{z}$, and which has Lipschitz modulus at $\bar{z}$ equal to $\|\nabla S(\bar{z})\| =\|(\nabla T^{-1}_{A,B}(\bar{z}))^{-1}\|\leq 1$.

Further assume that $\bar{z}=(\bar{x},0)\in A\cap B$. We shall show that $S$ coincides with $T_{A,B}$ around $\bar{z}$. First we note that since $S$ is a localization at $\bar{z}$ for $\bar{z}$, by definition, there exist neighborhoods $U$ and $V$ of $\bar{z}$ such that
\begin{equation}\label{eq:local}
(\forall z\in U)\quad (T_{A,B}^{-1})^{-1}(z)\cap V =Sz.
\end{equation}
Now set $\delta>0$ such that $\ball{\bar{z}}{\delta}\subseteq U$ and $\ball{\bar{z}}{2\delta}\subseteq V$. Applying \cref{l:1cvx}\ref{l:1cvx_stable} gives
\begin{equation}\label{eq:2expan}
(\forall z\in \ball{\bar{z}}{\delta})\quad T_{A,B}z \subseteq \ball{\bar{z}}{2\delta} \subseteq V.
\end{equation}
As $T_{A,B}\subseteq (T_{A,B}^{-1})^{-1}$, combining \eqref{eq:local} with \eqref{eq:2expan} gives that
\begin{equation}
(\forall z\in \ball{\bar{z}}{\delta})\quad T_{A,B}z =T_{A,B}z\cap V\subseteq (T_{A,B}^{-1})^{-1}(z)\cap V =Sz,
\end{equation}
and since $T_{A,B}z \neq \varnothing$ and $Sz$ is a singleton, the above inclusion must be an equality.
This yields $T_{A,B}=S$ on $\ball{\bar{z}}{\delta}$, as was claimed. 
We therefore deduce that $T_{A,B}$ is single-valued and locally Lipschitz on $\ball{\bar{z}}{\delta}$ with modulus at $\bar{z}$ equal to $\ell :=\|\nabla T_{A,B}(\bar{z})\| =\|\nabla S(\bar{z})\| =\|(\nabla T^{-1}_{A,B}(\bar{z}))^{-1}\| \leq 1$ satisfying \eqref{eq:dimesion bound}. This completes the proof.
\end{proof}

A closer inspection of the proof of \cref{t:stable} shows that it actually proves \emph{$Q$-linear} convergence of the DRA in a special case. Recall that a sequence $(z_n)_\nnn$ is said to converge \emph{$Q$-linearly} to $\bar{z}$ with rate $\kappa \in \left[0, 1\right[$ if 
\begin{equation}
\limsup_{n\to\infty} \frac{\|z_{n+1} -\bar{z}\|}{\|z_n -\bar{z}\|} \leq \kappa.
\end{equation}	

\begin{corollary}[Local $Q$-linear convergence of the DRA]
\label{c:Qlinear}
Let $\bar{z} :=(\bar{x}, 0) \in A\cap B$, and suppose that $X =\RR$ and that $f$ is twice strictly differentiable at $\bar{x}$ with $f'(\bar{x}) \neq 0$.
Then there exists $\delta \in \RPP$ such that $T_{A,B}$ is a single-valued contraction mapping on $\ball{\bar{z}}{\delta}$ with $T_{A,B}(\ball{\bar{z}}{\delta}) \subset \ball{\bar{z}}{\delta}$.   
Furthermore, for any starting point $z_0 \in \ball{\bar{z}}{\delta}$, the DR sequence $(z_n)_\nnn$ converges $Q$-linearly to $\bar{z}$ with  rate 
\begin{equation}
\kappa :=\frac{1}{\sqrt{1 +|f'(\bar{x})|^2}}.
\end{equation}
\end{corollary}	
\begin{proof}
By applying \cref{t:stable} to $X =\RR$, there exists $\delta \in \RPP$ such that $T_{A,B}$ is single-valued and locally Lipschitz continuous on $\ball{\bar{z}}{\delta}$ with modulus at $\bar{z}$ equal to $\kappa :=\|\nabla T_{A,B}(\bar{z})\| =1/\sqrt{1 +|f'(\bar{x})|^2} <1$. From the definition of the Lipschitz modulus at $\bar{z}$, we have
\begin{equation}
\label{e:modulus}
\limsup_{z \to \bar{z},\ z' \to \bar{z}} \frac{\|T_{A,B}z -T_{A,B}z'\|}{\|z -z'\|} \leq \kappa <1.
\end{equation}
Let $\kappa' \in \left]\kappa, 1\right[$. 
Then, by shrinking $\delta$ if necessary, we have
\begin{equation}
\label{e:contraction}
(\forall z \in \ball{\bar{z}}{\delta})(\forall z' \in \ball{\bar{z}}{\delta})\quad \|T_{A,B}z -T_{A,B}z'\| \leq \kappa'\|z -z'\|,
\end{equation}
and hence $T_{A,B}$ is a (single-valued) contraction mapping on $\ball{\bar{z}}{\delta}$.
Substituting $z' =\bar{z}$ and noting that $T_{A,B}\bar{z} =\bar{z}$ yield 
\begin{equation}
(\forall z \in \ball{\bar{z}}{\delta})\quad \|T_{A,B}z -\bar{z}\| \leq \kappa'\|z -\bar{z}\|,
\end{equation}
which implies that $T_{A,B}(\ball{\bar{z}}{\delta}) \subseteq \ball{\bar{z}}{\kappa'\delta}\subset \ball{\bar{z}}{\delta}$ 
and that the DRA sequence $(z_n)_{\nnn}$ converges to $\bar{z}$ whenever $z_0 \in \ball{\bar{z}}{\delta}$. 
Now since $z_n\to\bar{z}$, the claimed $Q$-linear rate follows from \eqref{e:modulus}.  
\end{proof}

\begin{remark}
Let $\bar z :=(\bar{x}, \bar{\rho}) \in \Fix T_{A,B}$ and suppose that $f$ is twice strictly differentiable at $\bar{x}$. 
By \cref{l:FixT}\ref{l:FixT_basic}, $(\bar{x}, f(\bar{x})) =(\bar{x}, 0) \in P_B(\bar{x}, -\bar{\rho})$ and so
\begin{equation}\label{eq:argmin fixed point}
\bar{x} \in \argmin_{y\in X} \tfrac{1}{2}\big(\|y -\bar{x}\|^2 +|f(y)+\bar{\rho}|^2\big).
\end{equation}
Differentiating the objective function twice gives
\begin{equation}
\Id +\nabla^2f(y)(f(y)+\bar{\rho}) +\nabla f(y)\nabla f(y)^\top.
\end{equation}
If $\bar{\rho} \neq 0$, then since $f(\bar{x}) =0$ and $\nabla f(\bar{x}) =0$ (\cref{l:FixT}\ref{l:FixT_Lipschitz}), 
the second order optimality condition yields
\begin{equation}\label{eq:second order opt}
\Id + \bar{\rho}\nabla^2f(\bar{x}) \succeq 0.
\end{equation}
Let us compare \eqref{eq:second order opt} to the assumption in \cref{t:stable}. 
The latter assumed that $\bar{\rho}\nabla^2f(\bar{x}) \succeq 0$ which is equivalent to
 \begin{equation}\label{eq:assumption t:stable}
 	 \Id + \bar{\rho}\nabla^2f(\bar{x}) \succeq \Id;
 \end{equation}
a condition which is stronger than \eqref{eq:second order opt}. 
Nevertheless, \eqref{eq:assumption t:stable} holds as soon as one of the following holds:
(i) $\bar{\rho} =0$, (ii) $\bar{\rho} \geq 0$ and $f$ is convex, or (iii) $\bar{\rho} \leq 0$ and $f$ is concave. 
In fact, when $\bar{\rho}\nabla^2f(\bar{x}) \succeq 0$ fails, unstable fixed points can arise 
as is the case in the following example.
\end{remark}

\begin{example}[An unstable fixed point]
\label{ex:unstable fixed point}
Consider $X =\RR$ and the function $f =\frac{1}{2}|\cdot|^2$. 
Appealing to Theorem~\ref{t:stable}, we deduce that $T_{A,B}$ is single-valued and locally Lipschitz around the point $(0,0)\in A\cap B$. However, $T_{A,B}$ is not locally Lipschitz around the point $\bar{z} =(0, -\frac{1}{2}) \in \Fix T_{A,B}\smallsetminus (A\cap B)$. To see this, let $\varepsilon \geq 0$ and consider the point $z_\varepsilon =(-\varepsilon, -\frac{1}{2})$. 
We have from \cref{l:ranT}\ref{l:ranT_eq} that
\begin{equation}
\label{e:Tze}
T_{A,B}z_\varepsilon =(0, -\tfrac{1}{2}) +P_B(-\varepsilon, \tfrac{1}{2}).
\end{equation}
Let $(y, f(y)) \in P_B(-\varepsilon, \frac{1}{2})$. Then \eqref{e:proj_gra} implies that
  \begin{equation}\label{eq:unstable fixed point}
    -\varepsilon =y+\left(f(y) -\tfrac{1}{2}\right) f'(y) =y +\left(\tfrac{1}{2}y^2 -\tfrac{1}{2}\right)y =\tfrac{1}{2}y^2(y +1).
  \end{equation} 
To show that $\bar{z}$ is in fact a fixed point, setting $\varepsilon =0$ in \eqref{eq:unstable fixed point}, we deduce that $y =0$ or $y =-1$. Further we observe that it cannot be the case that $y =-1$ since
  \begin{equation}
    \left\|(0,f(0)) -(0,\tfrac{1}{2})\right\|^2 =\tfrac{1}{4} <\tfrac{5}{4} =\left\|(-1, f(-1)) -(0,\tfrac{1}{2})\right\|^2,
  \end{equation}
and so we conclude that $P_B(0, \frac{1}{2}) =(0,0)$, which together with \eqref{e:Tze} gives 
$T_{A,B}\bar{z} =T_{A,B}z_0 =(0, -\frac{1}{2}) =\bar{z}$ and hence $\bar{z} \in \Fix T_{A,B}$.

Now, to see that $\bar{z}$ is not stable (in the sense of Lipschitz continuity of $T_{A,B}$), consider the point $z_\varepsilon =(-\varepsilon, -\frac{1}{2})$  can be made arbitrarily close to $\bar{z}$ by choosing $\varepsilon>0$ sufficiently small. For all $\varepsilon\approx 0$, the optimality condition \eqref{eq:unstable fixed point} has only one solution at $y\approx -1$. But this implies that 
  \begin{equation}
    T_{A,B}z_\varepsilon \approx (0, -\tfrac{1}{2}) +(-1, \tfrac{1}{2}) =(-1,0),
  \end{equation}
and consequently that $\|T_{A,B}z_\varepsilon -T_{A,B}\bar{z}\| \approx 1$ while $\|z_\varepsilon-\bar{z}\| =\varepsilon$, thus $T_{A,B}$ is not locally Lipschitz around $\bar{z}$. 
Note that it does not contradict \cref{t:stable} since the condition $\bar{\rho}f''(\bar{x})\geq 0$ is not satisfied.\qed
\end{example}
In a later example (Example~\ref{ex:p norm}), we show that in the setting of Example~\ref{ex:unstable fixed point} the DRA is globally convergent.

Recall that a sequence $(z_n)_\nnn$ is said to converge \emph{$R$-linearly} to a point $\bar{z}$ 
if there exist constants $\eta \in \RP$ and $\kappa \in \left[0, 1\right[$ such that
\begin{equation}
(\forall\nnn)\quad \|z_n - \bar{z}\| \leq \eta\kappa^n.
\end{equation}
Clearly the notion of $Q$-linear convergence implies $R$-linear convergence.

To complement the results in this section, we deduce following $R$-linear convergence result using existing results in the literature. Note that, in contrast to setting of \cref{t:stable}, the following result only applies to fixed points at which $\nabla f$ is nonsingular.

\begin{proposition}[Local $R$-linear convergence of the DRA]
\label{p:linear}
Let $\bar{z} :=(\bar{x},0) \in A\cap B$, and suppose that $f$ is continuously differentiable around $\bar{x}$ with $\nabla f(\bar{x}) \neq 0$. 
Then there exists $\delta>0$ such that, for any starting point ${z_0 \in \ball{\bar{z}}{\delta}}$, the DR sequence $(z_n)_\nnn$ converges $R$-linearly to a point in $A\cap B$.   
\end{proposition}
\begin{proof}
By assumption, $f$ is continuously differentiable on $U$ from some a neighborhood $U$ of $\bar{x}$. Define a function $G\colon U\times \RR \to \RR\colon (x, \rho) \mapsto f(x) -\rho$ and let $D :=\{0\} \subseteq \RR$.
Then $U\times \RR$ is a neighborhood of $(\bar{x}, 0)$, $G$ is a $C^1$ mapping, $D$ is a closed convex subset of $\RR$, 
\begin{subequations}
\begin{gather}
B\cap (U\times \RR) =\menge{(x, \rho) \in U\times \RR}{G(x, \rho) \in D},\quad\text{~and~} \\ 
(\forall \nu \in \RR)\quad \nabla G(\bar{x}, 0)^*\nu =0 \iff (\nabla f(\bar{x})\nu, -\nu) =0 \iff \nu =0.  
\end{gather}
\end{subequations}
In view of \cite[Definition~10.23(b)]{RW98}, $B$ is \emph{amenable} at $(\bar{x}, 0)$ and hence \emph{superregular} at $(\bar{x}, 0)$ by \cite[Proposition~4.8]{LLM09}. Moreover, the normal cones to $A$ and $B$ can be described, respectively, by \cite[Proposition~1.2]{Mor06} and \cite[Example~6.8]{RW98} as
\begin{subequations}
\begin{align}
N_A(\bar{x}, 0) &=N_X(\bar{x})\times N_{\{0\}}(0) =\{0\} \times \RR \subseteq X\times \RR,\quad\text{~and} \\
N_B(\bar{x}, 0) &=\menge{\nabla G(\bar{x}, 0)^*\nu}{\nu \in \RR} =\menge{(\nabla f(\bar{x})\nu, -\nu)}{\nu \in \RR}.
\end{align}
\end{subequations}
Since it is assumed that $\nabla f(\bar{x}) \neq 0$, it follows that $N_A(\bar{x}, 0)\cap (-N_B(\bar{x}, 0)) =\{0\}$, that is, to say that $\{A,B\}$ is \emph{strongly regular} at $(\bar{x}, 0)$. The assumptions of \cite[Theorem~4.3]{Pha16} (or \cite[Corollary~5.22]{DP16}) are thus satisfied, from which the result follows.
\end{proof}

To conclude this section, we note that \cref{t:stable} applies in situations when does not Proposition~\ref{p:linear}. In a subsequent section, we shall revisit the following example. 

\begin{example}[A stable fixed point]
Consider the function $f =\frac{1}{2}\|\cdot\|^2$ and the point $\{0\} =A\cap B \subseteq \Fix T_{A,B}\subseteq X\times\mathbb{R}$. Then $f$ does not satisfy the assumptions of \cref{p:linear} at $(\bar{x}, \bar{\rho}) =(0, 0)$ because $0 =f(0) =\nabla f(0)$. 
Nevertheless, as $f$ is twice continuously differentiable at $\bar{x}=0$ with $\nabla^2f=\Id$, \cref{t:stable} still applies and shows that the DR operator is single-valued and Lipschitz continuous around $(0, 0)$.\qed
\end{example}

\section{A Lyapunov-type approach to convergence}
\label{s:lyapunov}
 In this section, we prove convergence of the DRA assuming the existence of a Lyapunov-type function which is assumed to possess the following properties on a subset of $X\times \RR$. In fact, our framework also provides a procedure for the construction of such a function. In practice, this mean that the candidate Lyapunov-type function can be concretely constructed and its properties easily checked. 

\noindent\fcolorbox{black}{myblue}{%
\begin{minipage}{0.98\textwidth}
\begin{assumption}
\label{a:V}
There exists a proper convex function $F\colon D \to \left(-\infty, +\infty\right]$ and a nonempty convex subset $D$ of $\dom F$ such that the following hold:
\begin{enumerate}
\item 
\label{a:V_diff} 
The subdifferential of $F$ satisfies
\begin{equation}
\label{e:V_diff}
(\forall x\in D)\quad \partial F(x) \supseteq 
\begin{cases}
\Menge{\frac{f(x)}{\|x^*\|^2}x^*}{x^* \in \partial^0\! f(x)} &\text{~if~} 0 \notin \partial^0\! f(x), \\
\{0\} &\text{~if~} f(x) =0.
\end{cases}
\end{equation}
\item 
\label{a:V_coer}
$F$ is coercive.
\item 
\label{a:V_cont}
$F$ is continuous on $D\cap f^{-1}(0)$.
\end{enumerate}
\end{assumption}
\end{minipage}%
}

The intuition behind \cref{a:V}, specifically \eqref{e:V_diff}, is the similar to that proposed in \cite{Ben15}. 
One seeks a function $V\colon D\times \RR \to \left[-\infty, +\infty\right]$ of the form
\begin{equation}
\label{e:V_form}
(\forall (x, \rho) \in D\times \RR)\quad V(x, \rho) =F(x) +\frac{1}{2}\rho^2
\end{equation}
such that for every $z :=(x, \rho) \in D\times \RR$, its level set at the point $z_+$ is tangent to $z-z_+$, where $z_+\in T_{A,B}z$.  To do so, we construct an $F$ satisfying \cref{a:V} by anti-subdifferentiating \eqref{e:V_diff}. An illustration of such a function is given in \cref{fig:level set}. In particular, if the function $f$ is strictly differentiable at $x\not\in(\nabla f)^{-1}(0)$, then \eqref{e:V_diff} becomes
\begin{equation}
\nabla F(x) =\frac{f(x)}{\|\nabla f(x)\|^2}\nabla f(x).
\end{equation}
and further, when $\dim X =1$, then the expression further simplifies to $F'=f/f'$. 

The two piecewise-defined cases in \eqref{e:V_diff} are consistent in the sense that, if $0\not\in\partial f(x)$ and $f(x)=0$, then both cases yield $0\in\partial F(x)$. The inclusion of the ``$f(x)=0$" case allows our analysis to include situations in which the ``$0\not\in\partial^0\!f(x)$" case has a remove discontinuity.

\begin{figure}[htb]
	\centering
	\begin{subfigure}[b]{0.475\textwidth}
	\centering    
	\begin{tikzpicture}
	\begin{axis}[domain=-0:6,
	y domain=-3:3,
	colormap={blackwhite}{gray(0cm)=(1); gray(1cm)=(0)} ,
	view={0}{90}]
	\addplot3[contour gnuplot={number=35,labels={false}}]{x+exp(-x)*10+y*y/2};
	\addplot[black,domain=0:3.7]  {0.1*pow(2.71828,x)-1};
	\addplot[black] {0};
	\end{axis}
	\end{tikzpicture}
	\caption{A contour plot of the Lyapunov function.}
	\end{subfigure}
	\hfill
    \begin{subfigure}[b]{0.45\textwidth}
	\centering    
	\begin{tikzpicture}[scale=0.85]
	\draw[fill,opacity=0.25] (2.3025,0) circle (2.3024);
	\fill[color=black!25,fill,opacity=0.5] (0,0) -- plot[domain=0:6] (\x,{sqrt(2)*sqrt((10)-\x-10*exp(-\x))}) -- (6,0) -- cycle;
	\fill[color=black!25,fill,opacity=0.5] (0,0) -- plot[domain=0:6] (\x,{-sqrt(2)*sqrt((10)-\x-10*exp(-\x))}) -- (6,0) -- cycle;
	\fill[color=black!70,fill,opacity=0.5] (1.31,0) -- plot[domain=1.32:3.774] (\x,{sqrt(2)*sqrt(4.00513-\x-10*exp(-\x))}) -- (3.774,0) -- cycle;
	\fill[color=black!70,fill,opacity=0.5] (1.31,0) -- plot[domain=1.32:3.774] (\x,{-sqrt(2)*sqrt(4.00513-\x-10*exp(-\x))}) -- (3.774,0) -- cycle;
	
	\draw[color=black,thick] plot[domain=-1:3.8] (\x,{0.1*exp(\x)-1});
	\draw[color=black,thick] plot[domain=-1:6] (\x,0);
	
	\newcommand{\DRseq}{(0.1,-0.89),(0.35,-1.75),(1.05,-2.46),(3.26,-0.85),(2.99,0.14),(2.52,0.38),(2.23,0.31),(2.12,0.14),(2.13,-0.02),(2.22,-0.1),(2.35,-0.04),(2.35,0.01),(2.32,0.02),(2.3,0.02)};
	\draw (2.302585,0) node[draw=black,circle,inner sep=0.5pt,fill=black!80,anchor=center] {};
	\coordinate (x0) at (0,0) ;
	\draw (x0) node[above left] {$z_0$};
	\foreach \x [count=\j] in \DRseq {
		\pgfmathsetmacro{\jj}{\j-1}
		\draw \x coordinate (x\j);
		\draw[black!80,-] (x\jj) -- (x\j);
		\draw (x\jj) node[draw=black!80,circle,inner sep=0.5pt,fill=black!80,anchor=center] {};
	};
	\coordinate (z) at (1.05,-2.46) ;
	\draw (z) node[below right] {$z$};
	\coordinate (z+) at (3.26,-0.85) ;
	\draw (z+) node[below right] {$z_+$};
	\coordinate (zbar) at (2.30258509299,0) ;
	\draw (zbar) node[below right] {$\bar{z}$};
	\end{tikzpicture}
	\caption{DR sequence with $z_0=(0,0)$, the level set $\lev{V(z_0)}V$  (light gray), the ball $\ball{\bar{z}}{\|z_0-\bar{z}\|}$ (gray), and the level set $\lev{V(z_+)}V$ (dark gray).\label{subfig:not nonexp}}
    \end{subfigure}
	\caption{A Lyapunov function, $V$, for $f(x)=\frac{1}{10}\exp(x)-1$ which guarantees global convergence of  the DRA to $\bar{z}:=(\ln (10),0)$ (see \cref{ex:expo}).  Note also that (b) shows that the DR operator is not nonexpansive. \label{fig:level set}}
\end{figure}
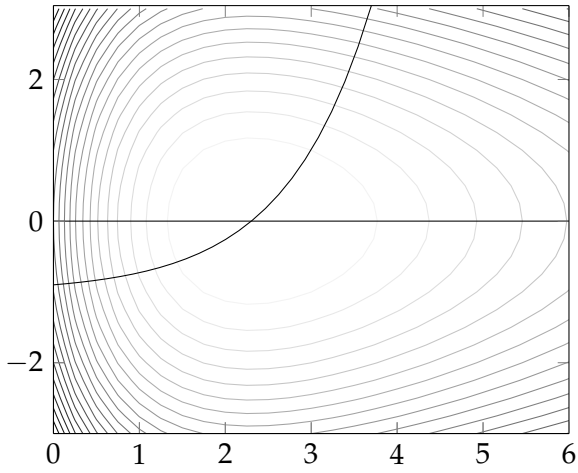
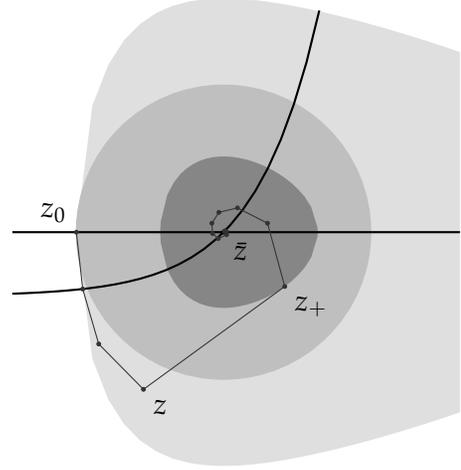

The following lemma investigates properties of the subdifferential of $V$ in \eqref{e:V_form}.
\begin{lemma} 
\label{l:partialV}
Let $z :=(x, \rho) \in \dom f \times \RR$, let $z_+ :=(x_+, \rho_+) \in T_{A,B}(x, \rho)$, and suppose that \cref{a:V}\ref{a:V_diff} holds.
The following assertions hold.
\begin{enumerate}
\item 
\label{l:partialV_formula}
$V$ is a proper convex function on $D\times \RR$ whose subdifferential is given by
\begin{equation}\label{eq:subdiff V product}
(\forall x \in \dom \partial F)\quad \partial V(x, \rho) =\partial F(x)\times \{\rho\}.
\end{equation}
\item
\label{l:partialV_0}
Suppose either: $x_+ \in D\smallsetminus (\partial^0\! f)^{-1}(0)$ and $f$ is Lipschitz continuous around $x_+$, or $x_+\in D\cap f^{-1}(0)$. Then 
then there exists $z'_+ :=(x'_+, \rho_+) \in \partial V(z_+)$ with $x'_+ \in \partial F(x_+)$ such that
\begin{equation}\label{eq:orthogonality condition}
\scal{z'_+}{z -z_+} =\scal{(x'_+, \rho_+)}{(x, \rho) -(x_+, \rho_+)} =0.
\end{equation}
\end{enumerate}
\end{lemma}
\begin{proof}
\ref{l:partialV_formula}: The function $V$ is proper and convex on $D\times\RR$, since $F$ and $\rho\mapsto\frac{1}{2}\rho^2$ are both proper and are convex on $D$ and $\RR$, respectively. 
Equation~\eqref{eq:subdiff V product} thus follows from \cite[Proposition~16.8]{BC11}.

\ref{l:partialV_0}: 
Let $x'_+ \in \partial F(x_+)$. By \ref{l:partialV_formula}, $z'_+ :=(x'_+, \rho_+) \in \partial F(x_+)\times \{\rho_+\} =\partial V(z_+)$.
We compute
\begin{equation}
\label{e:scal}
\scal{z'_+}{z -z_+} =\scal{(x', \rho_+)}{(x, \rho) -(x_+, \rho_+)} 
=\scal{x_+'}{x -x_+} -\rho_+f(x_+),
\end{equation}
using the fact that $\rho_+ =\rho +f(x_+)$ from \cref{l:DRstep}.
We consider two cases.

\emph{Case~1}: Suppose $x_+ \in D\smallsetminus (\partial^0\! f)^{-1}(0)$ and $f$ is Lipschitz continuous around $x_+$. 
Again by \cref{l:DRstep}, $x_+ =x -\rho_+x_+^*$ for some $x_+^* \in \partial^0\! f(x_+)$. 
Then $x_+^* \neq 0$.
Setting $x'_+ :=\frac{f(x_+)}{\|x_+^*\|^2}x_+^*$, it follows from \eqref{e:scal} that
\begin{equation}
\scal{z'_+}{z -z_+}  =\scal{\tfrac{f(x_+)}{\|x_+^*\|^2}x_+^*}{\rho_+x_+^*} -\rho_+f(x_+) =0. 
\end{equation}

\emph{Case~2}: Suppose $x_+\in D\cap f^{-1}(0)$. Then, by \cref{a:V}\ref{a:V_diff}, $x_+' :=0 \in \partial F(x_+)$ 
which completes the proof.
\end{proof}  
 
We are now ready to give our main result which analyses the Douglas--Rachford algorithm using the proposed Lyapunov-type function. We remark that the assumption in \eqref{e:stableD} with $n_0=0$ holds, in particular, in the setting of Section~\ref{s:stability of DRA} and could be formulated as ``the DRA is stable on $D$''.
\begin{theorem}[Convergence of the DRA]
\label{t:gencvg}
Suppose that \cref{a:V}\ref{a:V_diff}--\ref{a:V_coer} holds, that $f$ is locally Lipschitz continuous on $D\smallsetminus f^{-1}(0)$,  
and that
\begin{equation}\label{e:stableD}
(\exists n_0 \in \NN)(\forall n \geq n_0)\quad x_n \in D \quad\text{and}\quad x_{n+1} \notin (\partial^0\! f)^{-1}(0)\smallsetminus f^{-1}(0).
\end{equation}
Then $(z_n)_\nnn$ is bounded and asymptotically regular, 
and each of its cluster points $\bar{z}$ satisfy $P_A\bar{z} \in A\cap B$.
Suppose additionally that $\overline{D}\cap f^{-1}(0)$ is a singleton, say $\{\bar{x}\}$, contained in $D$ and that \cref{a:V}\ref{a:V_cont} holds. 
Then the following assertions hold:
\begin{enumerate}
\item 
\label{t:gencvg_noSlater}
The DR sequence $(z_n)_\nnn$ converges to a point $\bar{z}$ such that $P_A\bar{z}=(\bar{x},0) \in A\cap B$.
\item
\label{t:gencvg_Slater}
If $0\not\in\partial^0\!f(\bar{x})$ and $f|_D$ is continuous at $\bar{x}$,
then $(z_n)_\nnn$ converges to a point $\bar{z} =(\bar{x}, 0) \in A\cap B$. 
\end{enumerate}
\end{theorem}
\begin{proof}
Without loss of generality, we can and do assume that $n_0 =0$.
According to \cref{l:partialV}\ref{l:partialV_0}, there is a sequence $(z'_{n+1})_\nnn$ satisfying
\begin{equation}
\label{e:0}
(\forall\nnn)\quad z'_{n+1} =(x'_{n+1}, \rho_{n+1}) \text{~with~} x'_{n+1} \in \partial F(x_{n+1})
\quad\text{and}\quad \scal{z'_{n+1}}{z_n -z_{n+1}} =0.
\end{equation}
This together with the convexity of $F$ and \eqref{e:Fensubdiff} yields
\begin{subequations}
\label{e:decreasing}
\begin{align}
(\forall\nnn)\quad V(z_n) -V(z_{n+1}) &=(F(x_n) -F(x_{n+1})) +\tfrac{1}{2}(\rho_n^2 -\rho_{n+1}^2) \\
&\geq \scal{x'_{n+1}}{x_n -x_{n+1}} +\rho_{n+1}(\rho_n -\rho_{n+1}) +\tfrac{1}{2}|\rho_n -\rho_{n+1}|^2 \\ 
&=\scal{z'_{n+1}}{z_n -z_{n+1}} +\tfrac{1}{2}|\rho_n -\rho_{n+1}|^2 \\ 
&=\tfrac{1}{2}(\rho_n -\rho_{n+1})^2 \geq 0. 
\end{align}
\end{subequations} 
It follows that the sequence $(V(z_n))_\nnn$ is nonincreasing. Since $V$ is coercive by \cref{a:V}\ref{a:V_coer} and \cref{l:sumcoer}, the sequence $(z_n)_\nnn$ is bounded (by \cref{f:coer}\ref{f:coer_lev}), and hence so too are $(x_n)_\nnn$ and $(\rho_n)_\nnn$. 
Moreover, from the coercivity of $V$ and \cref{f:coer}\ref{f:coer_inf}, it transpires that $(V(z_n))_\nnn$ is bounded below and therefore convergent.
As a result, \eqref{e:decreasing} implies that $\rho_n -\rho_{n+1} \to 0$ and, by \cref{l:DRstep}, $f(x_n) =\rho_n -\rho_{n-1} \to 0$ as $n \to +\infty$, 
which combined with the boundedness of $(\rho_n)_\nnn$ yields
\begin{equation}
0 \leq \|x_n -x_{n+1}\|^2 \leq (f(x_n) -f(x_{n+1}))(f(x_n) +f(x_{n+1}) +2\rho_n) \to 0,
\end{equation}
and then $x_n -x_{n+1} \to 0$ as $n \to +\infty$. 
Therefore, $\|z_n -z_{n+1}\|^2 =\|x_n -x_{n+1}\|^2 +|\rho_n -\rho_{n+1}|^2 \to 0$,
which gives the asymptotic regularity of $(z_n)_\nnn$.  
Now \cref{l:1cvx}\ref{l:1cvx_cluster} completes the claim on cluster points of the sequence $(z_n)_\nnn$.

\ref{t:gencvg_noSlater}:
We first prove convergence of the sequence $(x_n)_\nnn$. 
Let $x$ be an arbitrary cluster point of the bounded sequence $(x_n)_\nnn$. 
Since $(\rho_n)_\nnn$ is bounded, there exists $\rho\in\mathbb{R}$ such that 
$z =(x, \rho)$ is a cluster point of $(z_n)_\nnn =((x_n, \rho_n))_\nnn$. 
It follows that $P_Az =(x,0) \in A\cap B$, hence $f(x) =0$ and consequently $x \in \overline{D}\cap f^{-1}(0) =\{\bar{x}\}$, which yields $x =\bar{x}$.
Thus, $\bar{x}$ is the unique cluster point of the bounded sequence $(x_n)_{\nnn}$, 
which implies that $(x_n)_\nnn$ converges to $\bar{x}$.
	
We now prove convergence of the sequence $(\rho_n)_\nnn$. 
To this end, let $\bar{\rho}$ denote an arbitrary cluster point of the bounded sequence $(\rho_n)_\nnn$. 
Then there is a subsequence $(z_{k_n})_\nnn =((x_{k_n}, \rho_{k_n}))_\nnn$ which converges to $\bar{z} =(\bar{x}, \bar{\rho})$. 
By \cref{a:V}\ref{a:V_cont}, $V$ is continuous at $\bar{z} =(\bar{x}, \bar{\rho})$ and so $V(z_{k_n}) \to V(\bar{z})$. 
Since $(V(z_n))_\nnn$ is nonincreasing (see \eqref{e:decreasing}), it follows that
\begin{equation}
\label{e:limV}
V(z_n) \to V(\bar{z}) =F(\bar{x}) +\tfrac{1}{2}\bar{\rho}^2 \quad\text{as}\quad n \to +\infty.
\end{equation} 
Combining with the continuity of $F$ gives
\begin{equation}
\rho_n^2 =2(V(z_n) -F(x_n))\to 2(V(\bar{z}) -F(\bar{x})) =\bar{\rho}^2 \quad\text{as}\quad n \to +\infty.
\end{equation} 
We thus deduce that $(\rho_n)_\nnn$ has at most two cluster point (namely, $\pm\bar{\rho}$). 
However, since $(\rho_n)_\nnn$ is bounded and asymptotically regular, 
\cite[Corollary~2.7]{BDM15} implies that the set of cluster points must be connected, and is therefore equal to $\{\bar{\rho}\}$. 
Therefore, $(\rho_n)_\nnn$ converges to $\bar{\rho}$ and the conclusion follows.

\ref{t:gencvg_Slater}: By \cref{l:DRstep}, 
$(\forall n \in \NN\smallsetminus \{0\})$ $x_{n-1} -x_n =\rho_n x_n^*$ with $x_n^* \in \partial^0\! f(x_n)$, 
and so 
\begin{equation}
\label{e:asymptotic}
\rho_n x_n^* \to 0 \quad\text{as~} n \to +\infty.
\end{equation}
By \ref{t:gencvg_noSlater}, $x_n \to \bar{x}$, $\rho_n \to \bar{\rho}$, and $(\bar{x}, 0) \in A\cap B$.
It thus suffices to show that $\bar{\rho} =0$. 
Suppose to the contrary that $\bar{\rho} \neq 0$. Then \eqref{e:asymptotic} yields $x_n^* \to 0$.
Since $x_n \to \bar{x}$ and since $f|_D$ is continuous at $\bar{x}$, we have that $x_n \stackrel{f}{\to} \bar{x}$
and hence $0 \in \partial^0\! f(\bar{x})$ due to \cref{l:robustness}. This contradicts the assumption that $0 \not\in \partial^0\! f(\bar{x})$ and completes the proof.
\end{proof}

We make the following observations regarding the proof of \cref{t:gencvg}.
\begin{remark}
\begin{enumerate}
\item 
To deduce boundedness of the sequence $(z_n)$ in \cref{t:gencvg}, \eqref{e:decreasing} shows $(V(z_n))$ to be nonincreasing and uses the coercivity of $V$ (which is equivalent to the assumed coercivity of $F$ in \cref{a:V}\ref{a:V_coer}). For this argument, it would suffice to assume that $V$ is \emph{weakly coercive} in the sense that $\inf V(D\times \RR) >-\infty$ and $(V(z_n))_\nnn$ is not a nonincreasing sequence as $z_n \in D\times \RR$, $\|z_n\| \to +\infty$.
\item
The assumption that \emph{``$\overline{D}\cap f^{-1}(0)$ is a singleton contained in $D$"} is satisfied, for instance, when the function $F$ satisfying \cref{a:V} is strictly convex on $\overline{D}\cap f^{-1}(0) \subseteq D$ (which holds, in particular, if $F$ is strictly convex on $D$). In this case, since $0\in F(x)$ whenever $f(x)=0$, it follows that $\overline{D}\cap f^{-1}(0)$ is the unique minimizer of $F$.
\end{enumerate}
\end{remark}

In the following corollary, we investigate linear convergence behavior of the DRA.
\begin{corollary}[Linear convergence of the DRA]
\label{c:lincvg}
Suppose that \cref{a:V} holds, that $f$ is locally Lipschitz continuous on $D\smallsetminus f^{-1}(0)$,
that $\overline{D}\cap f^{-1}(0) =\{\bar{x}\} \subseteq D$, 
and that
\begin{equation}\label{e:stableD2}
(\exists n_0 \in \NN)(\forall n \geq n_0)\quad x_n \in D \quad\text{and}\quad x_{n+1} \notin (\partial^0\! f)^{-1}(0)\smallsetminus f^{-1}(0).
\end{equation}
Then the following assertions hold:
\begin{enumerate}
\item 
\label{c:lincvg_R}
If $f$ is continuously differentiable around $\bar{x}$ with $\nabla f(\bar{x}) \neq 0$, then the DR sequence $(z_n)_\nnn$ converges $R$-linearly to $\bar{z} =(\bar{x}, 0) \in A\cap B$.
\item
\label{c:lincvg_Q}
If $X =\RR$ and $f$ is twice strictly differentiable at $\bar{x}$ with $f'(\bar{x}) \neq 0$, then the DR sequence $(z_n)_\nnn$ converges $Q$-linearly to $\bar{z} =(\bar{x}, 0) \in A\cap B$ with rate 
\begin{equation}\label{eq:Qlinear rate kappa}
\kappa :=\frac{1}{\sqrt{1 +|f'(\bar{x})|^2}}.
\end{equation}
\end{enumerate}
\end{corollary}
\begin{proof}
\ref{c:lincvg_R}:~Applying \cref{t:gencvg}\ref{t:gencvg_Slater} yields that $z_n\to\bar{z}=(\bar{x},0)\in A\cap B$. 
Since $\nabla f(\bar{x})\neq 0$, $R$-linear convergence of $(z_n)_{\nnn}$ follows from  \cref{p:linear}.

\ref{c:lincvg_Q}:~From \ref{c:lincvg_R} it follows, in particular, that $z_n\to\bar{z}=(\bar{x},0)$. 
Since $X =\mathbb{R}$ and $f'(\bar{x}) \neq 0$, $Q$-linear convergence of $(z_n)_{\nnn}$ with rate $\kappa$ given by \eqref{eq:Qlinear rate kappa} follows from \cref{c:Qlinear}.
\end{proof}

 Note that the \emph{stability condition} given by \eqref{e:stableD} in \cref{t:gencvg} (resp.\ \eqref{e:stableD2} in \cref{c:lincvg}) can be simplified in the following setting.
 \begin{remark}[Stability condition]\label{r:gencvg}
	Suppose that the function $f$ satisfies $(\partial^0\!f)^{-1}(0)\subseteq f^{-1}(0)$. Then condition \eqref{e:stableD} in \cref{t:gencvg} simplifies to
 	\begin{equation}
 	(\exists n_0 \in \NN)(\forall n \geq n_0)\quad x_n \in D.
 	\end{equation}
    The same simplification is also valid for \eqref{e:stableD2} in the context of \cref{c:lincvg}.
 \end{remark}

\section{Examples}
\label{s:examples}
In this section, we give some illustrative examples of instances in which either \cref{t:gencvg} or \cref{c:lincvg} can be utilized to deduce global/local convergence. In the following examples, we take it for granted that the convexity of the Lyapunov-type function $F$ is recognized using \cref{f:cvx}.

Our first two examples involve functions not possessing critical points.
\begin{example}[Global $Q$-linear convergence for the linear function]
Suppose that $X =\RR$ and that $f(x) :=\alpha x -\beta$, where $\alpha \in \RR\smallsetminus \{0\}$ and $\beta \in \RR$. This example is actually a convex feasibility problem, since both sets are affine lines.
The function $f$ is continuously differentiable (and hence locally Lipschitz) on $D :=\dom f =\RR$, $D\cap f^{-1}(0) =f^{-1}(0) =\{\beta/\alpha\}$, and $f' =\alpha \neq 0$. The function $F$ defined by
\begin{equation}
F(x) :=\int \frac{f(x)}{f'(x)}dx =\int \left(x -\frac{\beta}{\alpha}\right)dx =\frac{1}{2}x^2 -\frac{\beta}{\alpha}x
\end{equation} 
satisfies \cref{a:V}. Thus, since $f$ satisfies \cref{c:lincvg}\ref{c:lincvg_Q}, the DRA thus globally $Q$-linearly converges with rate $\kappa =1/\sqrt{1 +|\alpha|^2}$.

Note that, letting  $\theta$ denote the angle between $A$ and $B$ (two lines in the Euclidean plane), we have that $\theta \in \left]0, \pi/2\right]$, $\tan\theta =|\alpha|$, and so
\begin{equation}\label{eq:fred}
\kappa =\frac{1}{\sqrt{1 +\tan^2\theta}} =\cos\theta.
\end{equation}
For the DRA, the rate given by \eqref{eq:fred} is precisely that which was obtained and shown to be sharp in \cite{BBNPW14}. Consequently, our $Q$-linear rate is also sharp.\qed
\end{example}

\begin{example}[Global $Q$-linear convergence for the exponential function]
	\label{ex:expo}
Suppose that $X =\RR$ and that $f(x) :=\alpha\exp(x) -\beta$ with $(\alpha, \beta) \in \RPP^2$. 
Then $f$ is continuously differentiable (and hence locally Lipschitz) on $D :=\dom f =\RR$, $D\cap f^{-1}(0) =f^{-1}(0) =\{\ln(\beta/\alpha)\}$, and $(\forall x \in D)$ $f'(x) =\alpha\exp(x) \neq 0$. 
In this case, the function $F$ defined by
\begin{equation}
F(x):=\int \frac{f(x)}{f'(x)}dx =\int \left(1 -\frac{\beta}{\alpha}\exp(-x)\right)dx =x +\frac{\beta}{\alpha}\exp(-x)
\end{equation} 
satisfies \cref{a:V}. 
The function $f$ is twice strictly differentiable at $\bar{x} :=\ln(\beta/\alpha)$ with $f'(\bar{x}) =\beta >0$.
According to \cref{c:lincvg}\ref{c:lincvg_Q}, the DRA globally $Q$-linearly converges with rate $\kappa =1/\sqrt{1 +\beta^2}$.\qed
\end{example}	

The following example was considered in \cite[Example~6.3]{BDNP16a} for $B$ equal to the epigraph rather than graph of $f$. Note that although it does not satisfy the local regularity conditions required to apply \cite{Pha16}, our analysis is nevertheless able to prove global convergence. Note also that \cref{ex:unstable fixed point} concerning  unstable fixed points is a special case.
\begin{example}[Global convergence for $f =\alpha\|\cdot\|^p$]
\label{ex:p norm}
Suppose $f =\alpha\|\cdot\|^p$ for $\alpha \in \RR\smallsetminus \{0\}$ and $p \in \left]0, +\infty\right[$.
Then $D :=\dom f =X$, $f^{-1}(0)=\{0\}$, and $f$ is continuously differentiable (and hence locally Lipschitz) on $X\setminus \{0\}$ with 
\begin{equation}
\partial^0\! f(x) =\partial f(x) =\{\nabla f(x)\} =\{\alpha p\|x\|^{p-2}x\} \not\ni 0 \text{~whenever~} x \neq 0.
\end{equation}  
Thus we have $(\partial^0\! f)^{-1}(0) \subseteq \{0\} =f^{-1}(0)$,
\begin{equation}
(\forall x \in X\smallsetminus \{0\})\quad 
\frac{f(x)}{\|\nabla f(x)\|}\nabla f(x) =\frac{\alpha\|x\|^p}{\alpha^2p^2\|x\|^{2p-2}}\alpha p\|x\|^{p-2}x =\frac{1}{p}x,
\end{equation}
and deduce that the functions $F(x) =\frac{1}{2p}\|x\|^2$ satisfies \cref{a:V}. Noting \cref{r:gencvg}, we deduce that the assumptions of \cref{t:gencvg} are therefore satisfied and global convergence of the DRA follows.\qed
\end{example}

The Lyapunov function constructed in \cref{ex:p norm} also applied in the following variation.

\begin{example}[Global convergence for $f =\alpha|\cdot|^p\sgn(\cdot)$]
	\label{ex:p norm sign}
Suppose that $X =\RR$ and that $f =\alpha|\cdot|^p\sgn(\cdot)$, where $\alpha \in \RR\smallsetminus \{0\}$ and $p \in \left]0, +\infty\right[$.
Then $D :=\dom f =X$, $f^{-1}(0)=\{0\}$, and $f$ is continuously differentiable (and hence locally Lipschitz) on $X\setminus \{0\}$ with
\begin{equation}
\partial^0\! f(x) =\partial f(x) =\{f'(x)\} =\{\alpha p|x|^{p-1}\} \not\ni 0 \text{~whenever~} x \neq 0.
\end{equation} 
Therefore, $(\partial^0\! f)^{-1}(0) \subseteq \{0\} =f^{-1}(0)$ and 
\begin{equation}
(\forall x \in X\smallsetminus \{0\})\quad 
\frac{f(x)}{f'(x)} =\frac{1}{p}|x|\sgn(x) =\frac{1}{p}x.
\end{equation}
We deduce that $F(x) =\frac{1}{2p}x^2$ satisfies \cref{t:gencvg} (noting \cref{r:gencvg}) and global convergence follows. 

Furthermore, when $p \leq 1$, it can be seen that $\partial^0\! f(0) =\varnothing$. In this case, \cref{t:gencvg}\ref{t:gencvg_Slater} applies to show that the DRA globally converges to $(0, 0) \in A\cap B$; see \cref{fig:DRA x13} for an illustration when $\alpha =3$ and $p =1/3$. On the other hand, when $p > 1$, we have $\alpha p|x|^{p-1} \to 0$ as $x \to 0$ and \cref{l:robustness} implies $0 \in \partial^0\! f(0)$. In this case, \cref{t:gencvg}\ref{t:gencvg_Slater} does not apply and, further, \cref{fig:DRA x3} suggests that the DRA may converge to a point $\bar{z} \notin A\cap B$ which has $P_A\bar{z} \in A\cap B$.\qed    
\end{example}

\begin{figure}[htb]
	\centering
	\begin{subfigure}[b]{0.475\textwidth}
		\centering
		\includegraphics[height=0.7\textwidth]{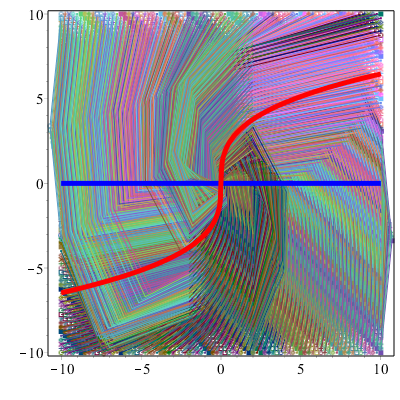}
		\caption{The trajectories of the DRA by starting point.\\[1.25em]}
	\end{subfigure}
	\begin{subfigure}[b]{0.475\textwidth}
		\centering     
		\includegraphics[height=0.7\textwidth]{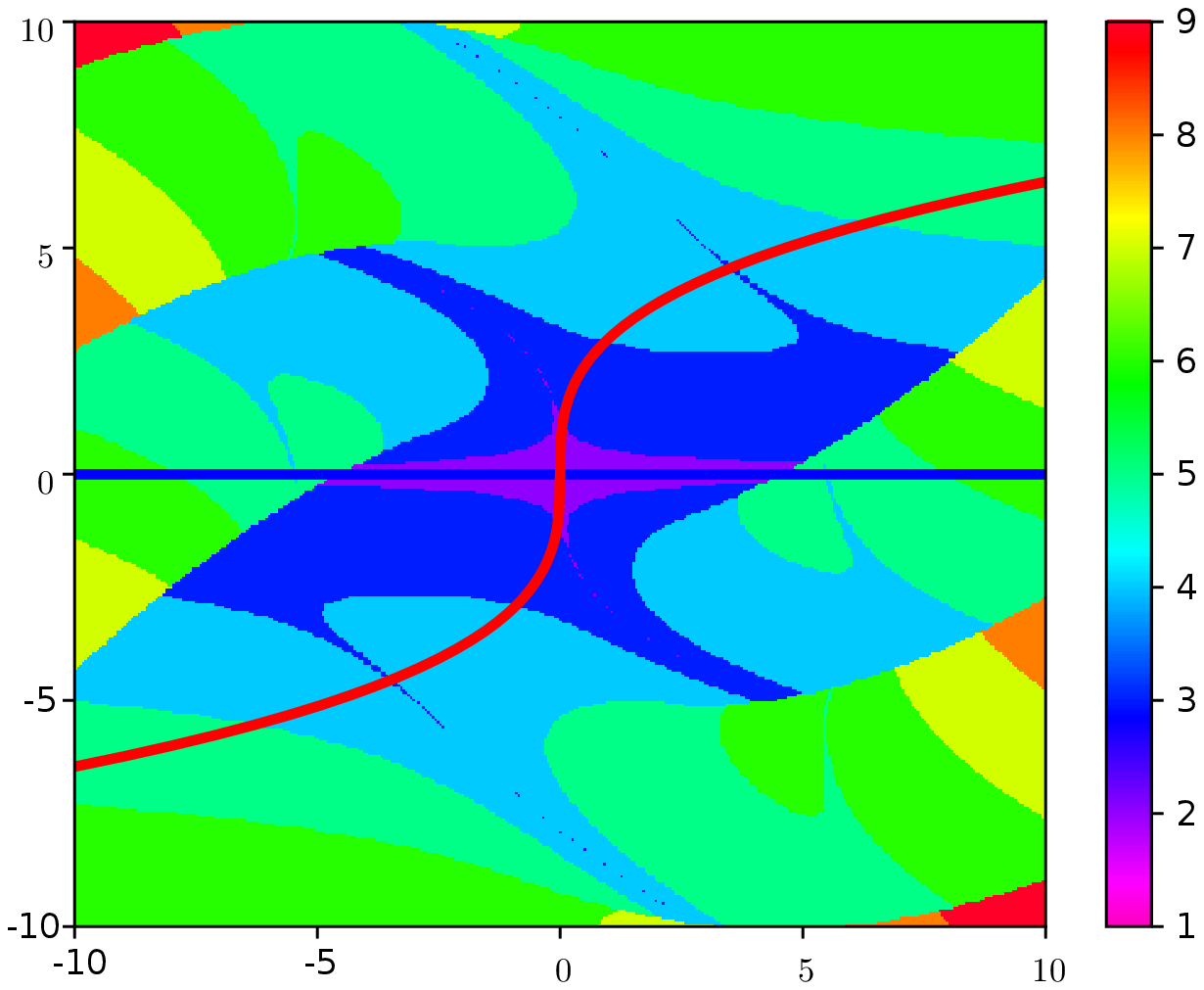}
		\caption{Number of iterations to reach within $10^{-6}$ of the solution by starting point.}
	\end{subfigure}
	\caption{Illustrations of the DRA for the function $f(x)=3\sqrt[3]{x}$ on $[-10,10]\times[-10,10]$. The corresponding feasibility problem has a unique solution with $A\cap B=\{(0,0)\}$.}\label{fig:DRA x13}
\end{figure}

\begin{figure}[htb]
	\centering
	\begin{subfigure}[b]{0.475\textwidth}
		\centering
		\includegraphics[height=0.7\textwidth]{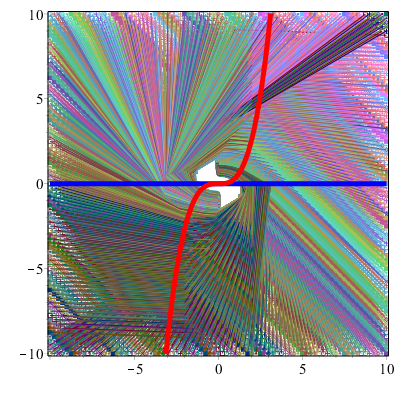}
      \caption{The trajectories of the DRA by starting point.\\[1.25em]}
	\end{subfigure}
	\begin{subfigure}[b]{0.475\textwidth}
		\centering     
		\includegraphics[height=0.7\textwidth]{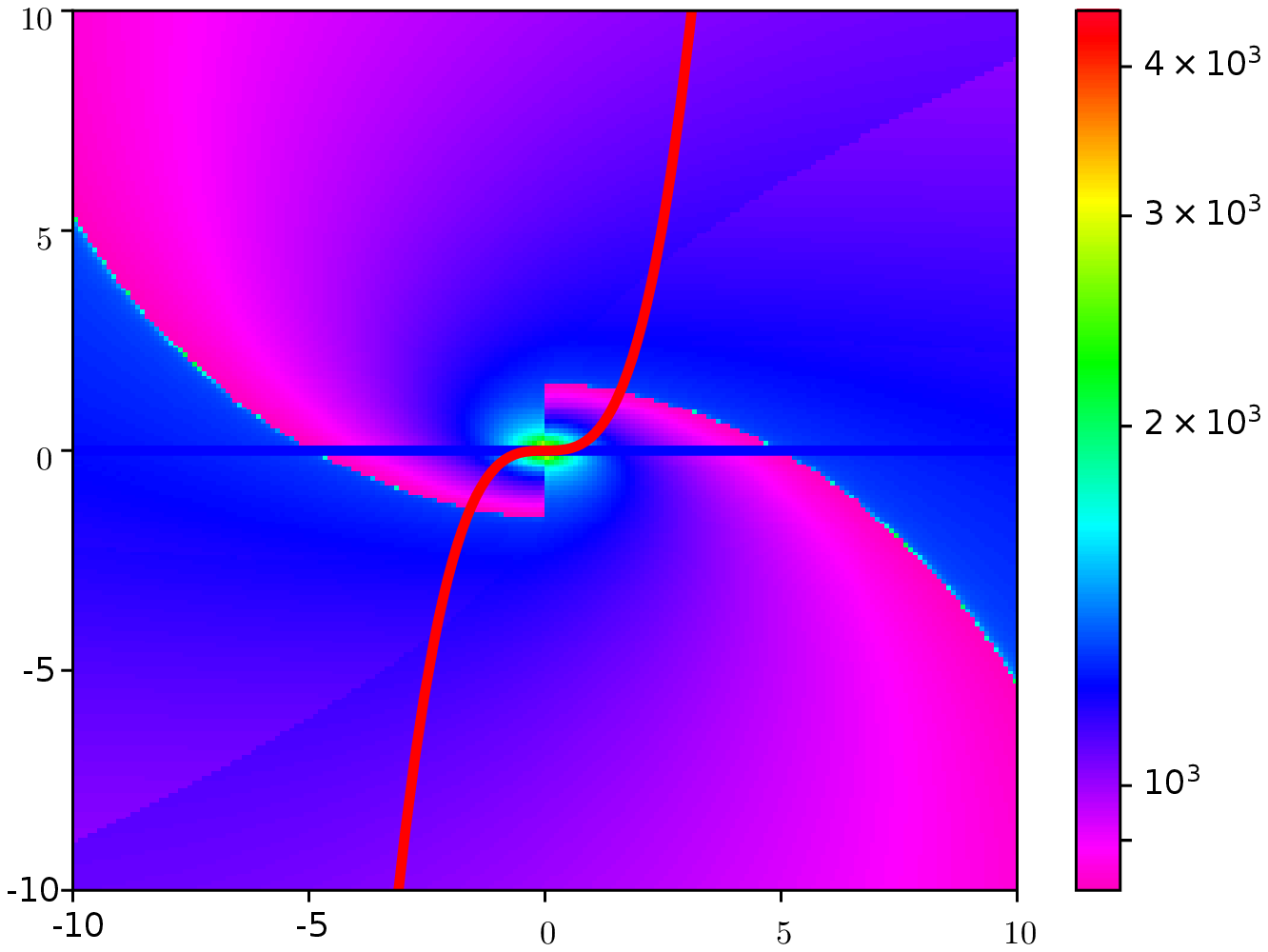}
      \caption{Number of iterations to reach within $10^{-6}$ of the solution by starting point.}
	\end{subfigure}
	\caption{Illustrations of the DRA for the function $f(x)=\frac{1}{3}x^3$ on $[-10,10]\times[-10,10]$. The corresponding feasibility problem has a unique solution with $A\cap B=\{(0,0)\}$.}\label{fig:DRA x3}
\end{figure}

\begin{example}[Benoist's Lyapunov function {\cite{Ben15}}]
Suppose that $X =\RR$ and let $f(x) =-\beta\sqrt{1 -x^2} +\alpha$, where $\beta \in \RPP$ and $\alpha \in \left]0, \beta\right[$. 
Up to symmetry, the case $\beta =1$ covers the setting of Benoist~\cite{Ben15}. 
We have that $\dom f =\left[-1, 1\right]$ and  
\begin{equation}
(\forall x \in \left]-1, 1\right[)\quad \partial^0\! f(x) =\partial f(x) =\{f'(x)\} =\left\{ \frac{\beta x}{\sqrt{1 -x^2}} \right\}
\end{equation} 
which yields continuous differentiability of $f$ on $\intdom f =\left]-1, 1\right[$ and that
\begin{subequations}
\begin{align}
F(x) &:=\int \frac{f(x)}{f'(x)} dx 
=\int \left( x -\frac{1}{x} +\frac{\alpha}{\beta}\frac{\sqrt{1-x^2}}{x} \right)dx \\
&\phantom{:}=\int \left( x -\big(1-\frac{\alpha}{\beta}\big)\frac{1}{x} -\frac{\alpha}{\beta}\frac{x}{1+\sqrt{1-x^2}} \right)dx \\
&\phantom{:}=\frac{1}{2}x^2 -\big(1 -\frac{\alpha}{\beta}\big)\ln|x| 
+\frac{\alpha}{\beta}\sqrt{1 -x^2} -\frac{\alpha}{\beta}\ln\big(1 +\sqrt{1 -x^2}\big).\label{eq:F benoist}
\end{align}
\end{subequations}
We note that the corresponding Lyapunov candidate function is therefore
\begin{equation}\label{eq:benoists Lyap}
V(x,\rho) =\frac{1}{2}x^2 -\big(1 -\frac{\alpha}{\beta}\big)\ln|x| 
+\frac{\alpha}{\beta}\sqrt{1 -x^2} -\frac{\alpha}{\beta}\ln\big(1 +\sqrt{1 -x^2}\big) +\frac{1}{2}\rho^2.
\end{equation}
When $\beta =1$, this is precisely the Lyapunov function obtained in \cite{Ben15}. 

Now let $\bar{x} \in f^{-1}(0) =\{\pm\sqrt{1 -(\alpha/\beta)^2}\} \subseteq \left]-1, 1\right[\smallsetminus \{0\}$ and set $\bar{z} :=(\bar{x}, 0) \in A\cap B$. Then $f'(\bar{x}) \neq 0$. 
Noting that $f'$ is strictly differentiable at $\bar{x}$, \cref{c:Qlinear} applies and thus there exists a $\delta>0$ such that, whenever $z_0 \in \ball{\bar{z}}{\delta}$, the DR sequence $(z_n)_\nnn$ is $Q$-linearly convergent to $\bar{z}$ with rate $\kappa$ given by
\begin{equation}
\kappa :=\frac{1}{\sqrt{1 +|f'(\bar{x})|^2}}  =\frac{\alpha}{\sqrt{\alpha^2 +\beta^2(\beta^2-\alpha^2)}}.
\end{equation}  
\qed     
\end{example}

\begin{example}[Global convergence for a nonsmooth, nonconvex function]
	\label{ex:nonsmooth nonconvex}
	Suppose $X =\RR$ and let $p\in \left]1, +\infty\right[$. 
	Consider the function $f\colon X\to \RR$ which is locally Lipschitz on $X$ and its symmetric subdifferential given by
	\begin{equation}
	f(x) := \begin{cases}
	x^p &\text{~if~} x \geq 0, \\
	x    &\text{~if~} x <0,
	\end{cases}\qquad
	\partial^0\!f(x) =\begin{cases}
	px^{p-1} &\text{~if~} x \geq 0, \\
	[0,1]   &\text{~if~} x =0, \\
	1   &\text{~if~} x <0,
	\end{cases}    
	\end{equation}	
	where we note that the nonconvex function $f$ is not smooth at $x = 0$, and that $(\partial^0\! f)^{-1}(0) =\{0\} =f^{-1}(0)$.
	
	An antiderivative of $x\mapsto \frac{f(x)}{\|\partial^0\! f(x)\|^2} \partial^0\! f(x)$ is given by $F$ where
	\begin{equation}
	F(x) := \begin{cases}
	\frac{1}{2p}x^2 &\text{~if~} x \geq 0, \\
	\frac{1}{2}x^2   &\text{~if~} x <0,
	\end{cases}
	\end{equation}
	and hence the function $F$ satisfies \cref{a:V} with $D=X$ (see \cref{fig:nonsmooth nonconvex}). From \cref{t:gencvg}, it follows that the DRA is globally convergent to a point $z$ such that $P_Az\in A\cap B$.\qed
\end{example}

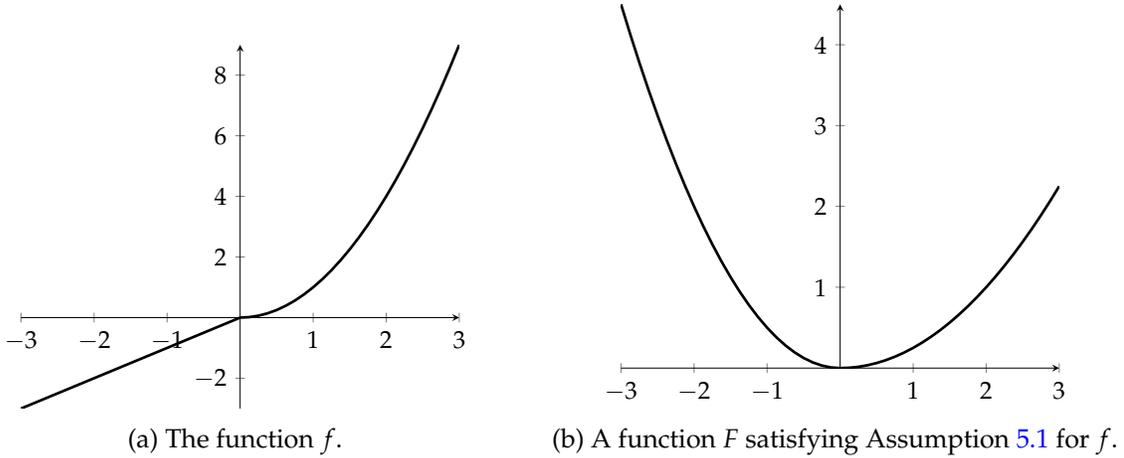
\begin{figure}[htb]
	\centering
	\begin{subfigure}[b]{0.475\textwidth}
		\centering    
		\begin{tikzpicture}[scale=0.85]
		\begin{axis}[axis lines=middle]
		\addplot[domain=0:3,very thick] {x^2};
		\addplot[domain=-3:0,very thick] {x};
		\end{axis}
		\end{tikzpicture}
		\caption{The function $f$.}
	\end{subfigure}
	\begin{subfigure}[b]{0.475\textwidth}
		\centering   
		\begin{tikzpicture}[scale=0.85]
		\begin{axis}[axis lines=middle]
		\addplot[domain=0:3,very thick] {1/4*x^2};
		\addplot[domain=-3:0,very thick] {1/2*x^2};
		\end{axis}
		\end{tikzpicture}
		\caption{A function $F$ satisfying \cref{a:V} for $f$.}
	\end{subfigure}
	\caption{The functions from Example~\ref{ex:nonsmooth nonconvex} with $p =2$.\label{fig:nonsmooth nonconvex}}
\end{figure}

Our concluding example compares the DRA with two other methods which can be used to find zeros of a function in this setting. Recall that the \emph{method of alternating projections (MAP)} generates sequences according to
\begin{equation}
(x_{n+1}, \rho_{n+1}) \in P_BP_A(x_n, \rho_n) =P_B(x_n, 0),
\end{equation}	  
where, applying Lemma~\ref{l:gra}, we have that the sequence $(x_n)$ satisfies 
\begin{equation}
x_n \in x_{n+1} +f(x_{n+1})\partial^0\! f(x_{n+1}),
\end{equation}
whenever $f$ is Lipschitz continuous around $x_{n+1}$ and, in particular, if $f$ is convex, then
\begin{equation}
x_n \in x_{n+1}+f(x_{n+1})\partial f(x_{n+1})\text{ whenever }x_{n+1}\in \intdom f.
\end{equation}
Recall also that \emph{Newton's method} for a (Fr\'echet) differentiable function $f$ is given by
  \begin{equation}\label{eq:newton}
  x_{n+1} = x_n-\left[\nabla f(x_n)\right]^{-1}f(x_n).
  \end{equation}
This iteration is only well defined provided that $\nabla f(x_n)$ is invertible for all $\nnn$. 
In particular, when $X =\RR$, \eqref{eq:newton} simplifies to
  \begin{equation}
  x_{n+1} = x_n-\frac{f(x_n)}{f'(x_n)}.
  \end{equation}
Under appropriate differentiability assumptions, Newton's method is locally \emph{quadratically convergent} for starting points sufficiently close to a zero, which is a faster rate than would be expected from a first-order method such as the DRA. In contrast, the global behavior of the Newton's method can be quite \emph{chaotic} and give rise to the so-called \emph{Newton fractals} \cite[Chapter~VII, Section~3]{Bar93}. The distinguishing feature of the DRA in the following is therefore that it converges from any starting point.

\begin{example}[DRA vs MAP vs Newton's method]
Suppose $X =\RR$ and consider the convex function $f\colon \RR\to\RR$ given by
\begin{equation}
f(x) :=\begin{cases}
-1 & \text{if }x \leq 0, \\
\frac{1}{2}x^2-1  & \text{if }x >0.
\end{cases}
\end{equation}
Then $f$ is continuously differentiable and hence also strictly differentiable with
\begin{equation}
f'(x) =\begin{cases}
0 & \text{if }x \leq 0, \\
x & \text{if }x >0.\\
\end{cases}
\end{equation}
Therefore, $(\forall x \in \RR)$ $\partial^0\!f(x) =\partial f(x) =\{f'(x)\}$. 
Moreover, $f^{-1}(0) =\{\sqrt{2}\}$, $A\cap B =\{(\sqrt{2},0)\}$ and $(f')^{-1}(0) =\left]-\infty, 0\right]$.
\begin{enumerate}
\item \emph{(Failure of the MAP):} The MAP is not globally convergent to an intersection point. In particular, for any starting point $z_0=(x_0,\rho_0)\in \RR^2$ with $x_0 \ll 0$, we have
     \begin{equation}
       z_1=P_BP_A(z_0)=P_B(x_0,0)=(x_0,-1).
     \end{equation}
  The MAP sequence starting at $z_0$ is therefore given by $z_n=(x_0,-1)\not\in A\cap B$ for every $n\geq 1$.

\item \emph{(Failure of Newton's Method):} Newton's method is also not globally convergent, \emph{a fortiori}, the Newton iteration is not globally well defined. Indeed, for any starting point $z_0 =(x_0,\rho_0)\in \left(-\infty, 0\right]\times\RR$, we have $f'(x_0)=0$ and hence \eqref{eq:newton} is not well defined (division by zero).

\item \emph{(Success of the DRA):} The DRA is globally convergent. To deduce this, note that $f$ is convex, $\dom f =\RR$ is open, and $\inf f(X)=-1 < \min\{0,\sup f(\dom f)\}$ so that, by \cref{c:bounded}, there exists an $n_0\in\mathbb{N}$ such that $f'(x_n) \neq 0$ for every $n\geq n_0$. We may therefore define the set $D$ to be
\begin{equation}
  D:=\dom f \smallsetminus (f')^{-1}(0) =\RR \smallsetminus \left]-\infty, 0\right] =\left]0, +\infty\right].
\end{equation}
 Now, as $f$ is continuously differentiable on $\RR$,  it is locally Lipschitz on $D =\left]0, +\infty\right]$. The function $F$ defined by 
 \begin{equation}
   (\forall x\in D)\quad F(x):=\int \frac{f(x)}{f'(x)}dx =\int \big(\frac{1}{2}x -\frac{1}{x}\big)dx =\frac{1}{4}x^2 -\ln(x),
 \end{equation}
satisfies \cref{a:V}. Furthermore, $\overline{D}\cap f^{-1}(0) =\{\sqrt{2}\} \subseteq D$, $f$ is twice strictly differentiable at $\sqrt{2}$ and $f'(\sqrt{2}) =\sqrt{2} \neq 0$. \cref{c:lincvg}\ref{c:lincvg_Q} therefore applies and we deduce that $z_n \to (\sqrt{2},0) \in A\cap B$ with $Q$-linear rate $1/\sqrt{3}$.\qed
\end{enumerate}
\end{example}

To conclude, we revisit \cref{ex:p norm sign} to compare the DRA with Newton's method.
\begin{example}[DRA vs Newton's method]
Suppose that $X =\RR$ and that $f =\alpha|\cdot|^p\sgn(\cdot)$ with $\alpha \in \RR\smallsetminus \{0\}$ and $p \in \left]0, 1/2\right]$.
As shown in \cref{ex:p norm sign}, $f'(x) =\alpha p|x|^{p-1}$ and $f(x)/f'(x) =\frac{1}{p}x$ whenever $x \neq 0$. 
\begin{enumerate}
\item \emph{(Failure of Newton's method):}~If $x_0 \neq 0$, then the Newton iteration for $f$ is given by
\begin{equation}
(\forall\nnn)\quad x_{n+1} =x_n -\frac{f(x_n)}{f'(x_n)} =x_n -\frac{1}{p}x_n =\big(1 -\frac{1}{p}\big)x_n \neq 0,
\end{equation}
and thus $(\forall\nnn)$ $x_n =(1 -1/p)^n x_0$.
Since $p \in \left]0, 1/2\right]$, we have $1 - 1/p \leq -1$ and the sequence $(x_n)$ is therefore not convergent. 
\item \emph{(Success of the DRA):}~The fact that the DRA converges globally to the unique solution $(0,0)\in A\cap B$ is a special case of \cref{ex:p norm sign}.\qed
\end{enumerate}
\end{example}

\subsection*{Acknowledgments}
This paper is dedicated to the memory of Jonathan M. Borwein and his enthusiasm for the DRA.
MND was partially supported by the Australian Research Council Discovery Project DP160101537. 
He wishes to acknowledge the hospitality and the support of D. Russell Luke during his visit to Universit\"at G\"ottingen. 
MKT was partially supported by the Deutsche Forschungsgemeinschaft Research Training Grant 2088.


\begin{thebibliography}{99}

\bibitem{AB13}
F.J.\ Arag\'on Artacho and J.M.\ Borwein, 
Global convergence of a nonconvex Douglas--Rachford iteration, 
\emph{Journal of Global Optimization}~57 (2013), 753--769.

\bibitem{ABT16}
F.J.\ Arag\'on Artacho, J.M.\ Borwein, and M.K.\ Tam,
Global behavior of the Douglas--Rachford method for a nonconvex feasibility problem,
\emph{Journal of Global Optimization}~65 (2016), 309--327.

\bibitem{Bar93}
M.F.\ Barnsley,
\emph{Fractals Everywhere}, Second edition,
Academic Press, 1993.

\bibitem{BBNPW14}
H.H.\ Bauschke, J.Y.\ Bello Cruz, T.T.A.\ Nghia, H.M.\ Phan, and X.\ Wang, 
The rate of linear convergence of the Douglas--Rachford algorithm for subspaces is the cosine of the Friedrichs angle, 
\emph{Journal of Approximation Theory}~185 (2014), 63--79.

\bibitem{BC11}
H.H.\ Bauschke and P.L.\ Combettes,
\emph{Convex Analysis and Monotone Operator Theory in Hilbert Spaces},
Springer, 2011.

\bibitem{BCL02} 
H.H.\ Bauschke, P.L.\ Combettes, and D.R.\ Luke,
Phase retrieval, error reduction algorithm, and Fienup variants: a view from convex optimization,
\emph{Journal of the Optical Society of America A}~19 (2002), 1334--1345.

\bibitem{BCL04} 
H.H.\ Bauschke, P.L.\ Combettes, and D.R.\ Luke,
Finding best approximation pairs relative to two closed convex sets in Hilbert spaces,
\emph{Journal of Approximation Theory}~127 (2004), 178--192.

\bibitem{BCN06} 
H.H.\ Bauschke, P.L.\ Combettes, and D.\ Noll, 
Joint minimization with alternating Bregman proximity operators,
\emph{Pacific Journal of Optimization}~2 (2006), 401--424.


\bibitem{BD16}
H.H.\ Bauschke and M.N.\ Dao,
On the finite convergence of the Douglas--Rachford algorithm for solving (not necessarily convex) feasibility problems in Euclidean spaces, 
\emph{SIAM Journal on Optimization}~ 27 (2017), 507--537.

\bibitem{BDM15}
H.H.\ Bauschke, M.N.\ Dao, and W.M.\ Moursi,
On Fej\'er monotone sequences and nonexpansive mappings,
\emph{Linear and Nonlinear Analysis}~1 (2015), 287--295.

\bibitem{BDM16}
H.H.\ Bauschke, M.N.\ Dao, and W.M.\ Moursi,
The Douglas--Rachford algorithm in the affine-convex case,
\emph{Operations Research Letters}~44 (2016), 379--382.

\bibitem{BDNP16a}
H.H.\ Bauschke, M.N.\ Dao, D.\ Noll, and H.M.\ Phan,
Proximal point algorithm, Douglas--Rachford algorithm 
and alternating projections: a case study,
\emph{Journal of Convex Analysis}~23 (2016), 237--261.

\bibitem{BDNP16}
H.H.\ Bauschke, M.N.\ Dao, D.\ Noll, and H.M.\ Phan,
On Slater's condition and finite convergence of the Douglas--Rachford algorithm for solving convex feasibility problems in Euclidean spaces,
\emph{Journal of Global Optimization}~65 (2016), 329--349.

\bibitem{BM16}
H.H.\ Bauschke and W.M.\ Moursi,
On the order of the operators in the Douglas--Rachford algorithm,
\emph{Optimization Letters}~10 (2016), 447--455.

\bibitem{BM17}
H.H.\ Bauschke and W.M.\ Moursi,
On the Douglas--Rachford algorithm,
\emph{Mathematical Programming, Series A}~164 (2017), 263--284.

\bibitem{Ben15}
J.\ Benoist,
The Douglas--Rachford algorithm for the case of the sphere and the line,
\emph{Journal of Global Optimization}~63 (2015), 363--380. 

\bibitem{BS11}
J.M.\ Borwein and B.\ Sims, 
The Douglas--Rachford algorithm in the absence of convexity,
In: \emph{Fixed-point Algorithms for Inverse Problems in Science and Engineering},
Springer, 2011, pp 93--109.

\bibitem{DP16}
M.N.\ Dao and H.M.\ Phan,
Linear convergence of projection algorithms,
2016, \href{http://arxiv.org/abs/1609.00341}{arXiv:1609.00341}.

\bibitem{DR56}
J.\ Douglas and H.H.\ Rachford,
On the numerical solution of heat conduction problems
in two and three space variables,
\emph{Transactions of the American Mathematical Society}~82 (1956), 421--439. 

\bibitem{Gil16}
O.\ Giladi,
A remark on the convergence of the Douglas--Rachford iteration in a non-convex setting,
2016, \href{http://arxiv.org/abs/1609.08751}{arXiv:1609.08751}.

\bibitem{HL13}
R.\ Hesse and D.R.\ Luke,
Nonconvex notions of regularity and convergence of fundamental
algorithms for feasibility problems,
\emph{SIAM Journal on Optimization}~23 (2013), 2397--2419. 

\bibitem{LLM09}
A.S.\ Lewis, D.R.\ Luke, and J.\ Malick,
Local linear convergence for alternating and averaged nonconvex projections,
\emph{Foundations of Computational Mathematics}~9 (2009), 485--513.

\bibitem{LM79}
P.-L.\ Lions and B.\ Mercier,
Splitting algorithms for the sum of two nonlinear operators,
\emph{SIAM Journal on Numerical Analysis}~16 (1979), 964--979. 

\bibitem{Mor06}
B.\ Mordukhovich, 
\emph{Variational Analysis and Generalized Differentiation I. Basic Theory}, 
Springer, 2006

\bibitem{Pha16}
H.M.\ Phan,
Linear convergence of the Douglas--Rachford method for two closed sets, 
\emph{Optimization}~65 (2016), 369--385.

\bibitem{NR16}
D.\ Noll and A.\ Rondepierre,
On local convergence of the method of alternating projections,
\emph{Foundations of Computational Mathematics}~16 (2016), 425--455.

\bibitem{RW98}
R.T.\ Rockafellar and R.J.-B.\ Wets, 
\emph{Variational Analysis}, 
Springer, 1998.

\bibitem{SM50}
J.\ Sherman and W.J.\ Morrison, 
Adjustment of an inverse matrix corresponding to a change in one element of a given matrix, 
\emph{Annals of Mathematical Statistics}~21 (1950), 124--127.

\bibitem{Sva11}
B.F.\ Svaiter,
On weak convergence of the Douglas--Rachford method,
\emph{SIAM Journal on Control and Optimization}~49 (2011), 280--287.

\end{thebibliography}
\end{document}